\documentclass[12pt]{article}
\usepackage[pdftex,bookmarksopen=true,bookmarks=true,unicode,setpagesize]{hyperref}
\hypersetup{colorlinks=true,linkcolor=black,citecolor=black}

\usepackage{hyperref}

\usepackage{amssymb, amsmath, amsthm}

\numberwithin{equation}{section}

\allowdisplaybreaks
\usepackage{cite}

\newcommand\R{\mathbb{R}}
\newcommand\N{\mathbb{N}}
\renewcommand\i{{\rm 1\kern -.3600em 1}}

\newtheorem{theorem}{Theorem}[section]
\newtheorem{corollary}[theorem]{Corollary}
\newtheorem{lemma}[theorem]{Lemma}
\newtheorem{proposition}[theorem]{Proposition}

\theoremstyle{remark}
\newtheorem{definition}[theorem]{Definition}
\theoremstyle{remark}
\newtheorem{example}[theorem]{Example}
\theoremstyle{remark}
\newtheorem{remark}[theorem]{Remark}

\newcommand{\vertiii}[1]{{\vert\kern-0.25ex\vert\kern-0.25ex\vert #1 
    \vert\kern-0.25ex\vert\kern-0.25ex\vert}}
\begin{document}

\vspace{-20mm}
\begin{center}{\Large \bf
Sheffer homeomorphisms of spaces of entire functions in infinite dimensional analysis}
\end{center}

{\large Dmitri Finkelshtein}\\ Department of Mathematics, Computational Foundry,  
Swansea University, Bay Campus, Swansea SA1 8EN, U.K.;\\
e-mail: \texttt{d.l.finkelshtein@swansea.ac.uk}\vspace{2mm}

{\large Yuri Kondratiev}\\ Fakult\"at f\"ur Mathematik, Universit\"at Bielefeld,
            33615~Bielefeld, Germany;\\
e-mail: \texttt{kondrat@mathematik.uni-bielefeld.de}\vspace{2mm}

{\large Eugene Lytvynov}\\ Department of Mathematics, Computational Foundry,  
Swansea University, Bay Campus, Swansea SA1 8EN, U.K.;\\
e-mail: \texttt{e.lytvynov@swansea.ac.uk}\vspace{2mm}

{\large Maria Jo\~{a}o Oliveira}\\ DCeT, Universidade Aberta,
            1269-001 Lisbon, Portugal; CMAFCIO, University of Lisbon, 1749-016 Lisbon, Portugal;\\
e-mail: \texttt{mjoliveira@ciencias.ulisboa.pt}\vspace{2mm}

{\large Ludwig Streit}\\ 
BiBoS, Universit\"at Bielefeld, Germany;\\ 
CIMA, Universidade da Madeira, Funchal, Portugal;\\
e-mail: \texttt{streit@uma.pt}

{\small
\begin{center}
{\bf Abstract}
\end{center}
\noindent
For certain Sheffer sequences $(s_n)_{n=0}^\infty$ on $\mathbb C$, Grabiner (1988) proved that, for each $\alpha\in[0,1]$, the corresponding Sheffer operator $z^n\mapsto s_n(z)$ extends to a linear self-homeomorphism of $\mathcal E^{\alpha}_{\mathrm{min}}(\mathbb C)$, the Fr\'echet topological space of entire functions of   order at most $\alpha$ and minimal type (when the order is equal to $\alpha>0$). In particular, every function $f\in \mathcal E^{\alpha}_{\mathrm{min}}(\mathbb C)$ admits a unique decomposition $f(z)=\sum_{n=0}^\infty c_n s_n(z)$, and the series converges in the topology of $\mathcal E^{\alpha}_{\mathrm{min}}(\mathbb C)$. Within the context of a complex nuclear space $\Phi$ and its dual space $\Phi'$, in this work we generalize  Grabiner's result to the case of Sheffer operators corresponding to Sheffer sequences on $\Phi'$. In particular, for $\Phi=\Phi'=\mathbb C^n$ with $n\ge2$, we obtain the multivariate  extension of Grabiner's theorem. Furthermore, for an Appell sequence on a general co-nuclear space $\Phi'$, we find a sufficient condition for the corresponding Sheffer operator to extend to a linear self-homeomorphism of $\mathcal E^{\alpha}_{\mathrm{min}}(\Phi')$ when $\alpha>1$. The latter result is new even in the one-dimensional case.
\noindent

 } \vspace{2mm}

\vspace{2mm}


\section{Introduction}

Let $(s_n)_{n=0}^\infty$ be a sequence of monic polynomials on $\mathbb C$.
(Although all polynomials sequences in this paper are assumed to be monic, this assumption is not essential for our purposes but  allows us to slightly simplify the notations.)
 Then  $(s_n)_{n=0}^\infty$ is called a {\it Sheffer sequence} if its (exponential) generating function has the form
\begin{equation}\label{xer5w5y43}
\sum_{n=0}^\infty \frac{u^n}{n!}\, s_n(z)=\frac{\exp[za(u)]}{r(a(u))}\,,\end{equation}
where $a(u)=\sum_{n=0}^\infty a_nu^n$ and $r(u)=\sum_{n=0}^\infty r_nu^n$ are formal power series in $u\in\mathbb C$ satisfying $a_0=0$, $a_1=1$ and $r_0=1$. In the special case  $r(u)\equiv 1$,  $(s_n)_{n=0}^\infty$ is called a {\it   sequence of polynomials of binomial type}, since it satisfies
$$s_n(z_1+z_2)=\sum_{k=0}^n\binom nk s_k(z_1)s_{n-k}(z_2).$$
In the special case  $a(u)\equiv u$, the Sheffer sequence $(s_n)_{n=0}^\infty$ is called an {\it Appell sequence}.
Modern {\it umbral calculus} (e.g.\ \cite{KRY,Roman}) studies Sheffer sequences and related operators. 

Let $\mathcal P(\mathbb C)$ denote the space of polynomials on $\mathbb C$. A  Sheffer sequence $(s_n)_{n=0}^\infty$ forms a basis in $\mathcal P(\mathbb C)$. Hence, we can define a linear operator $\mathfrak S$ acting on $\mathcal P(\mathbb C)$ by
$ \mathfrak S z^n=s_n(z)$.  The operator $\mathfrak S$ is called a {\it Sheffer operator}. In the special case where $(s_n)_{n=0}^\infty$ is a sequence of polynomials  of binomial type, the operator $\mathfrak S$ is called an {\it umbral operator}. The umbral, and more general Sheffer operators play a fundamental role in umbral calculus.

Grabiner  \cite{Grabiner} determined Banach and Fr\'echet spaces of entire functions between which Sheffer operators (in particular, umbral operators) can be extended as continuous linear operators, or even linear homeomorphisms. 

Let us now briefly recall one of the central results of \cite{Grabiner}. Recall that an entire function $f:\mathbb C\to\mathbb C$ is  of order at most $\alpha>0$ and minimal type (when the order is equal to $\alpha$) if $f$ satisfies the estimate  
$$\sup_{z\in\mathbb C}|f(z)|\exp(-t|z|^\alpha)<\infty\quad \text{for each }t>0.$$
We denote by $\mathcal E^{\alpha}_{\mathrm{min}}(\mathbb C)$ the linear space of such functions. This space is endowed with a natural Fr\'echet topology. Note that $\mathcal E^{\alpha_1}_{\mathrm{min}}(\mathbb C)\subset \mathcal E^{\alpha_2}_{\mathrm{min}}(\mathbb C)$ for $\alpha_1<\alpha_2$. We further denote $\mathcal E^{0}_{\mathrm{min}}(\mathbb C):=\bigcap_{\alpha>0}\mathcal E^{\alpha}_{\mathrm{min}}(\mathbb C)$, the Fr\'echet topological space of entire functions of order 0 (the lower index $\mathrm{min}$ in the notation $\mathcal E^{0}_{\mathrm{min}}(\mathbb C)$ being kept just for consistency with the case $\alpha>0$).
  Note that $\mathcal P(\mathbb C)$ is a dense subset of $\mathcal E^{\alpha}_{\mathrm{min}}(\mathbb C)$ for each $\alpha\ge0$. 
The following result is Theorem (3.13) and part of Theorem (6.6) in \cite{Grabiner}.

\begin{theorem}[\cite{Grabiner}]\label{xtew4u} Assume that functions $a(u)$ and $r(u)$ are holomorphic on a neighborhood of zero. Let a  Sheffer  sequence $(s_n)_{n=0}^\infty$ be defined by \eqref{xer5w5y43}. Then, for each $\alpha\in[0,1]$, the Sheffer operator $\mathfrak S$  corresponding  to $(s_n)_{n=0}^\infty$ extends  by continuity to a linear self-homeomorphism of the space $\mathcal E^{\alpha}_{\mathrm{min}}(\mathbb C)$. In particular, each function $f\in \mathcal E^{\alpha}_{\mathrm{min}}(\mathbb C)$ admits a unique representation
\begin{equation}\label{ftyqyrd5i}
f(z)=\sum_{n=0}^\infty c_n s_n(z),\quad c_n\in\mathbb C,\end{equation}
where the series on the right-hand side of formula \eqref{ftyqyrd5i} converges in the topology of $\mathcal E^{\alpha}_{\mathrm{min}}(\mathbb C)$.
\end{theorem}

Many extensions of umbral calculus to the case of polynomials of several, or even infinitely many variables were discussed e.g.\ in \cite{BBN,Brown,DBLR,Parrish,Reiner,Roman_multivariate,Ueno}, for a longer list of such papers see the introduction to \cite{DBLR}. However, in the multivariate case, no analog of Grabiner's results was proved for  Sheffer sequences.

In infinite dimensional (stochastic) analysis, one often meets examples of  sequences of polynomials defined on a co-nuclear space. More precisely, one considers 
 a Gel'fand triple $\Phi\subset \mathcal H_0\subset \Phi'$, where $\Phi$ is a nuclear space that is topologically (i.e., densely and continuously) embedded into a Hilbert space $\mathcal H_0$, and $\Phi'$ is the dual of the space $\Phi$ (i.e., a co-nuclear space) and the dual pairing between elements of  $\Phi'$ and $\Phi$ is given by a continuous extension of the scalar product in $\mathcal H_0$. Usually, $\mathcal H_0=L^2(\R^d,dx)$ with $d\in\mathbb N$, while $\Phi$ is either $\mathcal S$, the Schwartz space of smooth rapidly decreasing functions on $\R^d$, or $\mathcal D$, the space of smooth functions on $\R^d$ with compact support. To study stochastic analysis related to a  probability measure on $\Phi'$ (called a {\it generalized stochastic process\/}), one uses sequences of polynomials on $\Phi'$.

In paper \cite{FKLO}, the definition of a Sheffer  sequence on $\Phi'$ was given and general properties of Sheffer sequences were studied. This research was motivated by numerous available examples of such sequences of polynomials: Hermite polynomials in Gaussian white noise analysis (e.g.\ \cite{BK,HKPS,HOUZ}), Charlier polynomials in analysis related to the Poisson point process (e.g.\ \cite{IK,Kondratievetal,KKO}), Laguerre polynomials in analysis related to the gamma random measure  
(e.g.\  \cite{KL,Kondratievetal}), Meixner polynomials in analysis related to the Meixner white noise measure (e.g.\ \cite{L1,L2}), falling factorials on $\Phi'$ appearing in the theory of point processes (e.g.\ \cite{BKKL,KK,KKO2}), and a class of sequences of polynomials on $\Phi'$ with a generating function of a certain exponential type 
appearing  in biorthogonal analysis related to a rather general probability measures on $\Phi'$ (e.g.\ \cite{ADKS,KSWY}). %Note that, for most of these theories, $\Phi$ is either $\mathcal S$ or $\mathcal D$.

We assume that $\Phi'$ is a complex co-nuclear space, in which case $\Phi'=\bigcup_{\tau\in T}\mathcal H_{-\tau}$,  a union of complex Hilbert spaces $\mathcal H_{-\tau}$. By analogy with \cite[Chapter 2, Section 5.4]{BK} or \cite[Section 2.2]{KSWY}, one can naturally define $\mathcal E_{\mathrm{min}}^\alpha(\Phi')$, the space of entire functions on $\Phi'$ of order at most $\alpha\ge0$ and minimal type (when the order is equal to $\alpha>0$).

%\footnote{
It should be noted that,
when the nuclear space $\Phi$ is countably-Hilbert (i.e., the set $T$ is countable),  functions from $\mathcal E_{\mathrm{min}}^\alpha(\Phi')$ are  entire on the space $\Phi'$ equipped with the inductive limit topology of the $\mathcal H_{-\tau}$ spaces, see Corollary~\ref{rtew45uwe} below. If, however, $T$ is not a countable set, functions from  $ \mathcal E_{\mathrm{min}}^\alpha(\Phi')$ are entire on each Hilbert space $\mathcal H_{-\tau}$ but it is not known whether they are continuous (hence entire) with respect to the inductive limit topology on  $\Phi'$.
%} 

One of the central results of Gaussian white noise analysis is an internal description of the Hida Test Space  \cite[Chapter 2, Section 5.4]{BK}, see also \cite[Section 3]{Lee} and \cite{KPS,MY}. By analogy with Theorem \ref{xtew4u}, this result can now be stated as follows: {\it The Sheffer operator corresponding to the  sequence of Hermite polynomials on $\Phi'$  extends by continuity to a linear self-homeomorphism of the space $\mathcal E^{2}_{\mathrm{min}}(\Phi')$.} (Note that Grabiner's paper \cite{Grabiner} does not deal with spaces $\mathcal E^{\alpha}_{\mathrm{min}}(\mathbb C)$ where $\alpha>1$.)

Furthermore, in paper \cite{KSWY} (see also\cite{KSW}), for a class of Sheffer polynomials on $\Phi'$ (which includes the Hermite and Charlier polynomials), a similar  result was proved for a smaller space of test functions  (which is nowadays referred to as the   Kondratiev Test Space, cf.\ \cite{Kondrat}). Again by analogy with Theorem \ref{xtew4u}, this result can now be stated as follows: {\it For each Sheffer sequence considered in \cite{KSWY}, the corresponding Sheffer operator  extends by continuity to a linear self-homeomorphism of the space $\mathcal E^{1}_{\mathrm{min}}(\Phi')$.} 

The  main result of the present paper (Theorem \ref{sew56uw}) is a direct extension of Theorem \ref{xtew4u}: {\it If the infinite dimensional analogs of the functions $a(u)$ and $r(u)$ satisfy an assumption generalizing that of Theorem \ref{xtew4u}, then for each $\alpha\in[0,1]$,
the corresponding Sheffer operator  extends by continuity to a linear self-homeomorphism of the space $\mathcal E^{\alpha}_{\mathrm{min}}(\Phi')$.}
In the special case $\Phi=\mathcal H_0=\Phi'=\mathbb C$, we recover Theorem \ref{xtew4u}. Furthermore, by choosing $\Phi=\mathcal H_0=\Phi'=\mathbb C^n$ with $n\ge2$, we obtain the {\it (finite-dimensional) multivariate  extension of 
Theorem \ref{xtew4u}}.

Using the ideas from  the proof of Theorem \ref{sew56uw}, we also obtain the following result (Theorem~\ref{ft7r8o5}) regarding Appell sequences on $\Phi'$: {\it Let $\alpha>1$ and assume that the infinite dimensional analog of the function $r(u)$ satisfies a certain assumption, which depends on $\alpha$. Then the  corresponding Sheffer operator  extends by continuity to a linear self-homeomorphism of $\mathcal E^{\alpha}_{\mathrm{min}}(\Phi')$.} This result contains the internal description of the Hida Test Space as a special case. It should be noted that this result is new even in the one-dimensional case $\Phi'=\mathbb C$.

The paper is organized as follows. In Section \ref{vrtew6u3}, we collect 
the required preliminaries of complex analysis on nuclear and co-nuclear spaces. The proofs of several new propositions appearing in this section are given in Appendix. In Subsection~\ref{zszawaz}, we recall the definitions  of nuclear and co-nuclear spaces and the construction of a Gel'fand triple. In Subsection~\ref{tss5waq53q},  we define and discuss holomorphic functions  on nuclear and co-nuclear spaces. 
 Here our standard reference for complex analysis on locally convex topological spaces  is  Dineen's book \cite{Dineen2}. 
In Subsection~\ref{cdrt6ue46}, we discuss the spaces $\mathcal E^{\alpha}_{\mathrm{min}}(\Phi')$. For each $\alpha>0$, we obtain a new useful representation of $\mathcal E^{\alpha}_{\mathrm{min}}(\Phi')$ as the projective limit of certain Banach spaces. This result develops further the theory from \cite[Chapter 2, Section 5.4]{BK} and \cite[Section 2.2]{KSWY}. 

In Section~\ref{ye6i4}, we discuss  the umbral and Sheffer operators. 
First, in Subsection~\ref{rtew4w3xx}, we  recall the required definitions and results from \cite{FKLO},  regarding Sheffer sequences on $\Phi'$. 
Next, in Subsection~\ref{qwrsdreddw}, we formulate and prove the main results of the paper.  

Finally, in Section \ref{a43wq7}, we discuss some examples of Sheffer sequences on $\Phi'$.  We show, in particular, that every Sheffer sequence on $\mathbb C$ satisfying the assumptions of Theorem~\ref{xtew4u} can be lifted  to a Sheffer sequence  on $\Phi'=\mathcal S'$ or $\mathcal D'$ that satisfies the assumption of our main result, Theorem \ref{sew56uw}. 
 
\section{Complex analysis on nuclear and co-nuclear spaces}\label{vrtew6u3}

\subsection{Preliminaries on complex nuclear and co-nuclear spaces}\label{zszawaz}

Let us first recall the definition of a nuclear space, for details see e.g.\ \cite[Chapter 14, Section~2.2]{BSU2}. Consider a family of real separable Hilbert spaces $(\mathcal H_{\tau,\R})_{\tau\in T}$, where $T$ is an arbitrary index set. Assume that,  for any $\tau_1,\tau_2\in T$, there exists a $\tau_3\in T$ such that $\mathcal H_{\tau_3,\R}\subset \mathcal H_{\tau_1,\R}$ and $\mathcal H_{\tau_3,\R}\subset \mathcal H_{\tau_2,\R}$ and both embeddings are continuous. Further assume that, for each $\tau_1\in T$, there exists a $\tau_2\in T$ such that $\mathcal H_{\tau_2,\R}\subset \mathcal H_{\tau_1,\R}$, and the embedding operator  is of  Hilbert--Schmidt class. Consider the set $\Phi_{\R}:=\bigcap_{\tau\in T}\mathcal H_{\tau,\R}$ and assume that $\Phi_{\R}$ is dense in each Hilbert space $\mathcal H_{\tau,\R}$.  We introduce in $\Phi_\R$ the projective limit topology of the $\mathcal H_{\tau,\R}$ spaces. Then the linear topological space $\Phi_\R$ is called {\it nuclear}.

Let us assume that, for some $\tau_0\in T$, each Hilbert space $\mathcal H_{\tau,\R}$ with $\tau\in T$ is continuously embedded into $\mathcal H_{0,\R}:=\mathcal H_{\tau_0,\R}$. We will call $\mathcal H_{0,\R}$ the {\it center space}.
 Let $\Phi'_\R$ denote the dual space of $\Phi_\R$ with respect to the center space  $\mathcal H_{0,\R}$, i.e., the dual pairing between $\Phi_\R'$ and $\Phi_\R$ is obtained by continuously extending the inner product on  $\mathcal H_{0,\R}$,
see  e.g.\ \cite[Chapter 14, Section~2.3]{BSU2} for details. The space $\Phi'_\R$ is often called {\it co-nuclear}. We will use the notation $(\omega,\xi)_{\mathcal H_{0,\R}}$ for the dual pairing between $\omega\in\Phi'_{\R}$ and $\xi\in\Phi_\R$.

By the Schwartz theorem (e.g.\ \cite[Chapter 14, Theorem~2.1]{BSU2}), $\Phi'_\R=\bigcup_{\tau\in T}\mathcal H_{-\tau,\R}$, where $\mathcal H_{-\tau,\R}$ denotes the dual space of $\mathcal H_{\tau,\R}$ with respect to the center space $\mathcal H_{0,\R}$. We endow $\Phi'_\R$ with the  topology of the  inductive limit of the $\mathcal H_{-\tau,\R}$ spaces. (Note that this is the locally convex inductive limit, i.e., the inductive limit is taken in the category of locally convex spaces and continuous linear maps.)
Thus, we obtain   the real {\it Gel'fand triple} 
$$
\Phi_\R=\operatornamewithlimits{proj\,lim}_{\tau\in T} \mathcal H_{\tau,\R}\subset\mathcal  H_{0,\R}\subset \operatornamewithlimits{ind\,lim}_{\tau\in T}\mathcal H_{-\tau,\R}=\Phi'_\R.$$

The above construction can now be  extended to the complex case. For each $\tau\in T$, we define the Hilbert space $\mathcal H_\tau$ as the complexification of $\mathcal H_{\tau,\R}$.   
(We assume that the inner product on $\mathcal H_\tau$ is linear in the first variable and antilinear in the second variable.) As a result, we obtain a complex nuclear space $\Phi$ that is the complexification of $\Phi_{\R}$.
Similarly, we define complex Hilbert spaces $\mathcal H_{-\tau}$ and their inductive limit $\Phi'$.
Each $\omega\in\Phi'$ determines a linear continuous functional on $\Phi$ by the formula
$\langle\omega,\xi\rangle:=(\omega,\overline\xi)_{\mathcal H_0}$,
where $\overline\xi$ denotes the complex conjugate of $\xi$.
Hence, $\Phi'$ is the dual of the complex nuclear space $\Phi$ (i.e., a complex co-nuclear space). Thus, we get the complex Gel'fand triple 
\begin{equation}\label{fd6re7itfrde}
\Phi=\operatornamewithlimits{proj\,lim}_{\tau\in T}\mathcal H_\tau\subset\mathcal H_0\subset  \operatornamewithlimits{ind\,lim}_{\tau\in T}\mathcal H_{-\tau}=\Phi'.\end{equation}

\begin{remark}
As the dual of the locally convex topological vector space $\Phi$, the space $\Phi'$ can be endowed with several standard topologies, in particular, the weak topology~$\sigma(\Phi',\Phi)$ (e.g.\ \cite[Chapter~II, Subsection~5.2]{SW}), 
the Mackey topology~$\tau(\Phi',\Phi)$ (e.g.\ \cite[Chapter~IV, Subsection~2.2]{SW}), and 
the strong topology~$\beta(\Phi',\Phi$) (e.g.\ \cite[Chapter~IV, Section~5]{SW}). In fact, the inductive limit topology on $\Phi'$, that we use in this paper, coincides with the Mackey topology~$\tau(\Phi',\Phi)$,  see e.g.\ \cite[Chapter IV, Subsection 4.4]{SW}. If the set $T$ is countable (i.e., $\Phi$ is a countably-Hilbert nuclear space), then the inductive limit/Mackey topology also coincides with the strong topology~$\beta(\Phi',\Phi$), see e.g.\ \cite[Theorem 4.16]{Benchel}.
Furthermore, in the latter case,  $\Phi'$ is a nuclear space itself, see  e.g.\ \cite[Chapter~IV, Subsection~9.6]{SW}.
\end{remark}

Below, we will denote by $\otimes$ the {\it tensor product} and by $\odot$ the {\it symmetric tensor product}. Let $n\in\mathbb N$.  Starting with Gel'fand triple \eqref{fd6re7itfrde}, one constructs its $n$th symmetric tensor power as follows:
$$\Phi^{\odot n}:=\operatornamewithlimits{proj\,lim}_{\tau\in T}\mathcal H_\tau^{\odot n}\subset \mathcal H_0^{\odot n}\subset \operatornamewithlimits{ind\,lim}_{\tau\in T}\mathcal H_{-\tau}^{\odot n}=:\Phi'{}^{\odot n},$$
see e.g.\ \cite[Section~2.1]{BK} for details. In particular, $\Phi^{\odot n}$ is a complex nuclear space and $\Phi'{}^{\odot n}$ is its dual.
We  will also define $\Phi^{\odot 0}=\mathcal H_0^{\odot 0}=\Phi'{}^{\odot 0}:=\mathbb C$.  The norm in each Hilbert space $\mathcal H_\tau^{\odot n}$ ($n\in\mathbb N$) will be denoted by $\|\cdot\|_\tau$. Similarly, we will use the notation $\|\cdot\|_{-\tau}$ for the norms in $\mathcal H_{-\tau}^{\odot n}$.
 
\subsection{Holomorphic mappings on complex nuclear and co-nuclear spaces}
\label{tss5waq53q}

Let us first recall some definitions and statements related to holomorphic mappings between locally convex spaces \cite{Dineen2}.
Let $E$ and $F$ be complex locally convex topological vector spaces. Let $U$ be an open subset of $E$. A mapping $f:U\to F$ is called {\it G\^ateaux holomorphic} if, for any $\xi\in U$, $\eta\in E$ and  $\Psi\in F'$, the $\mathbb C$-valued function 
$$z\mapsto \langle \Psi, f(\xi+z\eta)\rangle\in\mathbb C$$
is holomorphic on some neighborhood of zero in $\mathbb C$. If additionally the mapping\linebreak[4]  $f:U\to F$ is continuous, then  $f$ is called  {\it holomorphic}.

We note that, if $F=\mathbb C$, the above definition of a G\^ateaux holomorphic function
$f:U\to\mathbb C$ means that, for any $\xi\in U$ and $\eta\in E$, the function $z\mapsto f(\xi+z\eta)\in\mathbb C$ is holomorphic on some neighborhood of zero in $\mathbb C$.

%In the case where $U$ is an open  subset of a locally convex space $E$ and $F$ is a normed vector space, a G\^ateaux holomorphic function $f:U\to F$ is holomorphic if and only if it is locally bounded \cite[Proposition~3.7]{Dineen2}.

A function $f:E\to\mathbb C$ that is holomorphic on a whole locally convex space $E$ is called {\it entire}.

Below we will use the spaces $\Phi$ and $\Phi'$ from the Gel'fand triple \eqref{fd6re7itfrde}. 
The space $\Phi'$, equipped with the inductive limit topology, is a locally convex space. By using the results of \cite[Section 3.1]{Dineen2},  it is not difficult to show that   each entire function $f:\Phi'\to\mathbb C$ has the property that, for each $\tau\in T$, the restriction $f\restriction_{ \mathcal H_{-\tau}}$ is an entire function on $\mathcal H_{-\tau}$. However, the converse statement is, generally speaking, not true.  We will say that a function $f:\Phi'\to\mathbb C$ is {\it restricted-entire on $\Phi'$} if its restriction  is entire on $\mathcal H_{-\tau}$ for each $\tau \in T$. 
(See Corollary~\ref{rtew45uwe} below which shows that, in the case of a countably-Hilbert nuclear space $\Phi$, the  functions from the spaces appearing in this paper are actually entire on $\Phi'$.)

For the proofs of the following two propositions, see Appendix.

\begin{proposition}\label{vytde6e6}
 Let  $f:\Phi'\to\mathbb C$ be  a restricted-entire function on $\Phi'$. Then there exist kernels $\varphi^{(n)}\in\Phi^{\odot n}$ such that, for all $\omega\in\Phi'$,
\begin{equation}\label{t7red6q5}
f(\omega)=\sum_{n=0}^\infty \langle\omega^{\otimes n},\varphi^{(n)}\rangle.
\end{equation}
Here $\omega^{\otimes0}:=1$.
\end{proposition}

Let $E,F$ be complex locally convex topological vector spaces. We denote by $\mathcal L(E,F)$ the space of continuous linear operators acting
from $E$ into $F$. We will also denote $\mathcal L(E):=\mathcal L(E,E)$.

\begin{proposition}\label{6ei650}  Let $\mathcal G$ be a separable complex Hilbert space. Let $U$ be an open neighborhood of zero in $\Phi$, and let 
$F:U\to\mathcal G$ be holomorphic. Then there exist  $U'$, an open neighborhood of zero in $\Phi$ that is a subset of $U$,  $\tau \in T$, and operators $A_n\in \mathcal L(\mathcal H_\tau^{\odot n},\mathcal G)$ ($n\in\mathbb N$) such that, for all $\xi\in U'$,
%\begin{equation}\label{r61ed644ed}
$ F(\xi)=F(0)+\sum_{n=1}^\infty A_n\xi^{\otimes n}$, 
%\end{equation}
 where the series %on the right-hand side of formula \eqref{r61ed644ed}
 converges in $\mathcal G$. Furthermore, for some $C_1\ge0$, we have 
$ \|A_n\|_{ \mathcal L(\mathcal H_\tau^{\odot n},\mathcal G)}\le C_1^n$ for all $n\in\mathbb N$.
\end{proposition}

The following two corollaries are now immediate.

\begin{corollary}\label{e6r656e3}
 Let  $U$ be an open neighborhood of zero in $\Phi$, and let  $f:U\to\mathbb C$ be holomorphic. Then there exist  $U'$, an 
open neighborhood of zero in $\Phi$ that is a subset of $U$,  $\tau \in T$, and $\rho^{(n)}\in\mathcal H_{-\tau}^{\odot n}$ ($n\in\mathbb N$) such that, for all $ \xi\in U'$,
$f(\xi)=f(0)+\sum_{n=1}^\infty \langle \rho^{(n)},\xi^{\otimes n} \rangle$.  
Furthermore,  for some  $C_2\ge0$, we have 
  $\|\rho^{(n)}\|_{-\tau}\le C_2^n$ for all $n\in\mathbb N$.
\end{corollary}

\begin{corollary}\label{tuqdftqq}
Let $U$ be an open neighborhood of zero in $\Phi$ and  let $F:U\to\Phi$ be holomorphic. Then there exist operators $A_n\in\mathcal L (\Phi^{\odot n},\Phi)$ ($n\in\mathbb N$) 
for which the following statements hold:

{\rm (a)} For each $\tau \in T$, there exists  $U_{\tau}$, an 
open neighborhood of zero in $\Phi$, that is a subset of $U$, such that, for all $\xi\in U_\tau$, 
%\begin{equation}\label{dfwtfwf}
$F(\xi)=F(0)+ \sum_{n=1}^\infty A_n\xi^{\otimes n}$,
%\end{equation}
where the series converges in the $\mathcal H_{\tau}$ space.

{\rm (b)} For each $\tau\in T$, there exist $\tau'\in T$ and $C_3\ge0$ such that, 
for each  $n\in\mathbb N$, 
$A_n$ extends  to a bounded  linear operator $A^\tau_n\in\mathcal L(\mathcal H_{\tau'}^{\odot n},\mathcal H_{\tau})$ with
$\|A^\tau_n\|_{\mathcal L(\mathcal H_{\tau'}^{\odot n},\mathcal H_{\tau})}\le C_{3}^n$.
\end{corollary}

\begin{remark}
Below we will drop the upper index $\tau$ in the notation $A_n^\tau$, hence we will use the same symbol $A_n$ for all extensions of the operator $A_n\in\mathcal L (\Phi^{\odot n},\Phi)$ to an operator from
 $\mathcal L(\mathcal H_{\tau'}^{\odot n},\mathcal H_{\tau})$. 
\end{remark}

Let us now fix arbitrary operators $A_n\in\mathcal L(\Phi^{\odot n},\Phi)$ ($n\in\mathbb N$) and consider the formal power series $F:\Phi\to\Phi$ given by
\begin{equation}\label{ewq432q4}
F(\xi)=F(0)+ \sum_{n=1}^\infty A_n\xi^{\otimes n},
\end{equation}
where $F(0)$ is a fixed element of $\Phi$. Note that, for a fixed $\xi\in\Phi$,   we get
$F(z\xi)=F(0)+\sum_{n=1}^\infty z^n A_n\xi^{\otimes n}$, 
which is a formal series in powers of $z\in\mathbb C$ with coefficients from $\Phi$. See \cite{FKLO} for further details on such formal power series.

Assume that the operators $A_n$ satisfy the following condition:
\begin{itemize}
\item[(QH)] For each 
$\tau\in T$, there exist $\tau'\in T$ and $C_{4}\ge0$ such that
$A_n\in\mathcal L(\mathcal H_{\tau'}^{\odot n},\mathcal H_\tau)$ and 
$\|A_n\|_{\mathcal L(\mathcal H_{\tau'}^{\odot n},\mathcal H_\tau)}\le C_{4}^n$ for all $n\in\mathbb N$, . 
\end{itemize}

\begin{definition} \label{64834w7}
We will say that a formal power series $F:\Phi\to\Phi$ given by \eqref{ewq432q4} is {\it quasi-holomorphic on a neighborhood of zero} if the operators $A_n\in\mathcal L(\Phi^{\odot n},\Phi)$ satisfy condition (QH).
\end{definition}

By Corollary \ref{tuqdftqq}, if $U$ is an open neighborhood of zero in $\Phi$ and a function\linebreak[4] $F:U\to\Phi$ is holomorphic, then it is also quasi-holomorphic.  Note, however, that, for a  sequence of operators $A_n$ satisfying the condition (QH), it may happen  that the $F(\xi)$ given by the convergent series \eqref{ewq432q4} belongs to $\Phi$ only for $\xi=0$.

\subsection{Spaces $\mathcal E^{\alpha}_{\mathrm{min}}(\Phi')$}\label{cdrt6ue46}

Let $\alpha>0$. For  $\tau\in T$ and  $l\in\N_0:=\{0,1,2,\dots\}$, we denote by $\mathcal E^\alpha_{l}(\mathcal H_{-\tau})$ the vector space of all entire functions $f:\mathcal H_{-\tau}\to\mathbb C$ that satisfy
\begin{equation}\label{d6e7i4}
\mathbf n_{\tau,l,\alpha}(f):=\sup_{\omega\in\mathcal H_{-\tau}}|f(\omega)|\exp(-2^{-l}\|\omega\|_{-\tau}^\alpha)<\infty.\end{equation}
The function $\mathbf n_{\tau,l,\alpha}(\cdot)$ determines a norm on $\mathcal E^\alpha_{l}(\mathcal H_{-\tau})$.  

We now define $\mathcal E^{\alpha}_{\mathrm{min}}(\Phi')$ as the set of all functions $f:\Phi'\to\mathbb C$ such that the res\-triction of $f$ to each Hilbert space $\mathcal H_{-\tau}$ ($\tau\in T$) belongs to $\mathcal E^\alpha_{l}(\mathcal H_{-\tau})$ for all $l\in\mathbb N_0$. We call $\mathcal E^{\alpha}_{\mathrm{min}}(\Phi')$ the {\it space of 
(restricted-)entire functions on $\Phi'$ of order at most $\alpha$ and minimal type (when the order is equal to $\alpha$)}.  We endow $\mathcal E^{\alpha}_{\mathrm{min}}(\Phi')$ with the topology of the projective limit: 
$$\mathcal E^{\alpha}_{\mathrm{min}}(\Phi'):= \operatornamewithlimits{proj\,lim}_{\tau\in T,\,l\in\N_0}\mathcal E^\alpha_{l}(\mathcal H_{-\tau}),$$
compare with \cite{BK,KSW,KSWY}.  For any $0<\alpha<\alpha'$, we obviously have
$\mathcal E^{\alpha}_{\mathrm{min}}(\Phi')\subset \mathcal E^{\alpha'}_{\mathrm{min}}(\Phi')$,
and the embedding is continuous. So, we also define
\begin{equation}\label{tye64w3w}
\mathcal E^{0}_{\mathrm{min}}(\Phi'):= \operatornamewithlimits{proj\,lim}_{\alpha>0}\mathcal E^{\alpha}_{\mathrm{min}}(\Phi').\end{equation}
%Functions from $\mathcal E^{0}_{\mathrm{min}}(\Phi')$ are called {\it (restricted-)entire on $\Phi'$ of minimal exponential growth}.

We will now present an alternative description  of $\mathcal E^{\alpha}_{\mathrm{min}}(\Phi')$. Recall that, by Proposition~\ref{vytde6e6}, each function $f\in \mathcal E^{\alpha}_{\mathrm{min}}(\Phi')$ with $\alpha\ge0$ has a representation \eqref{t7red6q5} with $\varphi^{(n)}\in\Phi^{\odot n}$. Let $\tau\in T$, $l\in\N_0$, and $\alpha>0$. Let $\mathbf E^\alpha_{l}(\mathcal H_{-\tau})$ denote the vector space of all entire functions $f:\mathcal H_{-\tau}\to\mathbb C$ of the form  \eqref{t7red6q5} with $\varphi^{(n)}\in\mathcal H_\tau^{\odot n}$ that satisfy
\begin{equation}\label{yufwrq76rf}
\|f\|_{\tau,l,\alpha}:=\sum_{n=0}^\infty (n!)^{1/\alpha}\, 2^{ln}\|\varphi^{(n)}\|_\tau<\infty. \end{equation}
The function $\|\cdot\|_{\tau,l,\alpha}$ determines a norm on  $\mathbf E^\alpha_{l}(\mathcal H_{-\tau})$, which makes the latter a Banach space.

For the proof of the following theorem, see Appendix. 

\begin{theorem}\label{tctfdr6tqde6}
 For each $\alpha>0$, the following equality of  topological spaces holds:
$$ \mathcal E^{\alpha}_{\mathrm{min}}(\Phi')= 
\operatornamewithlimits{proj\,lim}_{\tau\in T,\,l\in\N_0}\mathbf E^\alpha_{l}(\mathcal H_{-\tau}).$$
\end{theorem}

\begin{remark}
The reader is advised to compare Theorem~\ref{tctfdr6tqde6} with \cite[Theorem 2.5]{KSWY}, which deals with the case $\alpha\in[1,2]$ and a  different collection of norms as compared with $\|\cdot\|_{\tau,l,\alpha}$.
\end{remark}

The following corollary is immediate.

\begin{corollary}\label{trsw5y647ew} Let $\alpha\in[0,\infty)$ and let  $f\in  \mathcal E^{\alpha}_{\mathrm{min}}(\Phi')$. Represent $f$ in the form \eqref{t7red6q5} (see Proposition~\ref{vytde6e6}). Then the series $\sum_{n=0}^\infty \langle\omega^{\otimes n},\varphi^{(n)}\rangle$ converges in the topology of $\mathcal E^{\alpha}_{\mathrm{min}}(\Phi')$. 
\end{corollary}

\begin{remark} By analogy with the proof of Theorem~\ref{tctfdr6tqde6}, it is not difficult to show that, for each $\alpha\in[0,\infty)$, $\mathcal E^{\alpha}_{\mathrm{min}}(\Phi')$ is a nuclear space.
\end{remark}

\begin{corollary}\label{rtew45uwe}
Assume $T=\mathbb N_0$ and for each $\tau\in\mathbb N_0$, the Hilbert space $\mathcal H_{\tau+1}$ is continuously embedded into $\mathcal H_{\tau}$ (equivalently, $\Phi$ is a countably-Hilbert nuclear space).
Let $\alpha\ge0$. Then,  each  function $f\in  \mathcal E^{\alpha}_{\mathrm{min}}(\Phi')$ is entire on the space $\Phi'$ equipped with the inductive limit topology.
\end{corollary}

The proof of Corollary \ref{rtew45uwe} is a modification of the proof of   \cite[Theorem 3.2]{KY}. We leave the details to the interested reader.

\section{Umbral and Sheffer operators}\label{ye6i4}

Let us first briefly recall some definitions and results from \cite{FKLO}. 

\begin{remark}
It should be noted that paper \cite{FKLO} deals with the real Gel'fand triple
$\mathcal D\subset L^2(\R^d,dx)\subset \mathcal D'$, where $\mathcal D$ is the nuclear space of all smooth compactly supported functions on $\R^d$. In fact, all the results of \cite{FKLO} can be 
immediately extended to the complexification of this real Gel'fand triple. Furthermore, it is easy to see that the  results of \cite{FKLO} that we will use in the present paper are true for an arbitrary complex Gel'fand triple \eqref{fd6re7itfrde}.
\end{remark}

\subsection{Sheffer  sequences on $\Phi'$}\label{rtew4w3xx}

A function $p:\Phi'\to\mathbb C$ is called a {\it polynomial on $\Phi'$} if
$p(\omega)=\sum_{k=0}^n\langle\omega^{\otimes k},\varphi^{(k)}\rangle$, where $\varphi^{(k)}\in\Phi^{\odot k}$, $k=0,1,\dots,n$, $n\in\mathbb N_0$. If $\varphi^{(n)}\ne 0$, one says that $p$ is a {\it polynomial of degree $n$}.
We denote by $\mathcal P(\Phi')$ the linear space of all polynomials on $\Phi'$.  

Assume that, for each $n\in\mathbb N_0$,  a mapping $P^{(n)}:\Phi'\to \Phi '{}^{\odot n}$
is of the form $P^{(n)}(\omega)=\sum_{k=0}^n U_{n,k}\,\omega^{\otimes k}$
 with $U_{n,k}\in\mathcal L(\Phi '{}^{\odot k},\Phi '{}^{\odot n})$. Furthermore, assume that, for each $n\in\mathbb N_0$, $U_{n,n}=\mathbf 1$, the identity operator on
$\Phi'{}^{\odot n}$. Then we call $(P^{(n)})_{n=0}^\infty$ a {\it sequence of  monic polynomials on $\Phi '$}.  

Note that  
\begin{equation*}
\langle P^{(n)}(\omega),\varphi^{(n)}\rangle=\langle\omega^{\otimes n},
\varphi^{(n)}\rangle+\sum_{k=0}^{n-1}\langle\omega^{\otimes k},
V_{k,n}\varphi^{(n)}\rangle,\quad \varphi^{(n)}\in\Phi ^{\odot n}.%\label{gtdyf7ti}
\end{equation*}
Here $V_{k,n}:=U_{n,k}^*\in \mathcal L(\Phi ^{\odot n},\Phi ^{\odot k})$, i.e., $V_{k,n}$ is the adjoint operator of $U_{n,k}$. In particular, $\langle P^{(n)}(\cdot),\varphi^{(n)}\rangle\in \mathcal P(\Phi')$. 
Note also that each $P^{(n)}(\omega)\in\Phi'^{\odot n}$ is completely determined by the values 
$\big(\langle P^{(n)}(\omega),\xi^{\otimes n}\rangle\big)_{\xi\in\Phi}$.

Let $(P^{(n)})_{n=0}^\infty$ be sequence of monic polynomials on $\Phi '$. Then each polynomial $p:\Phi'\to\mathbb C$ of degree $n$ has a unique representation $p(\omega)=\sum_{k=0}^n\langle P^{(k)}(\omega),\varphi^{(k)}\rangle$ with $\varphi^{(k)}\in \Phi^{\odot k}$. 
Define  a linear mapping $\mathfrak S:\mathcal P(\Phi')\to\mathcal P(\Phi')$ by 
\begin{equation}\label{xreas4q4aas}
\mathfrak S \langle \omega ^{\otimes n},\varphi^{(n)}\rangle:=\langle P^{(n)}(\omega),\varphi^{(n)}\rangle,\quad \varphi^{(n)}\in\Phi^{\odot n},\ n\in\mathbb N_0.\end{equation}
Then $\mathfrak S$ is a bijective mapping. 

Let $(P^{(n)})_{n=0}^\infty$ be a sequence of monic polynomials  on $\Phi'$.
We say that  $(P^{(n)})_{n=0}^\infty$ is of {\it binomial type} if, for any $n\in\mathbb N$ and any $\omega,\zeta\in\Phi'$, 
$$P^{(n)}(\omega+\zeta)=\sum_{k=0}^n\binom{n}{k}P^{(k)}(\omega)\odot P^{(n-k)}(\zeta).$$
Here $P^{(k)}(\omega)\odot P^{(n-k)}(\zeta)\in \Phi'{}^{\odot n}$ is  the symmetrization of 
 $P^{(k)}(\omega)\otimes P^{(n-k)}(\zeta)\in \Phi'{}^{\otimes n}$. 
By \cite[Theorem~4.1]{FKLO}, a sequence of monic polynomials, $(P^{(n)})_{n=0}^\infty$\,,  is of binomial type if and only if it has the generating function 
\begin{equation}
\sum_{n=0}^\infty\frac{1}{n!}\langle P^{(n)}(\omega),\xi^{\otimes n}\rangle=\exp\big[\langle\omega,A(\xi)\rangle\big],\quad\omega\in\Phi',\label{Eq23}
\end{equation}
where
\begin{equation}\label{ft7r}
A(\xi)=\sum_{k=1}^\infty A_k\xi^{\otimes k},\end{equation}
with
$A_k\in\mathcal{L}(\Phi^{\odot  k},\Phi)$, $k\in\mathbb N$, and
$A_1=\mathbf 1$, the identity operator on  $\Phi$.

\begin{remark}\label{dr65w83s}
For $A(\xi)$ as in formula \eqref{ft7r},  there exists a formal power series
$
B(\xi)=\sum_{k=1}^\infty B_k\xi^{\otimes k}$
with  $B_k\in\mathcal{L}(\Phi^{\odot  k},\Phi)$, $k\in\mathbb N$, and
$B_1=\mathbf 1$, that is the {\it compositional inverse of} $A(\xi)$, i.e., 
$A(B(\xi))=B(A(\xi))=\xi$.
 Hence, formula \eqref{Eq23} implies
\begin{equation}\label{gu7r}
\sum_{n=0}^\infty\frac{1}{n!}\langle P^{(n)}(\omega),\big(B(\xi)\big)^{\otimes n}\rangle=\exp\big[\langle\omega,\xi\rangle\big],\quad\omega\in\Phi'.
\end{equation}
\end{remark}

Let $(S^{(n)})_{n=0}^\infty$ be a sequence of monic polynomials on $\Phi'$. We call $(S^{(n)})_{n=0}^\infty$ a {\it Sheffer sequence} if it has the generating function
\begin{equation}\label{tyqde6iq}
\sum_{n=0}^\infty\frac{1}{n!}\langle S^{(n)}(\omega),\xi^{\otimes n}\rangle
=\frac{\exp[\langle\omega,A(\xi)\rangle]}{\rho(A(\xi))},\end{equation}
where $A(\xi)$ is as in formula \eqref{ft7r} and 
\begin{equation}\label{edfuyr7i}
\rho(\xi)=\sum_{n=0}^\infty \langle\rho^{(n)},\xi^{\otimes n}\rangle,\end{equation}
 with $\rho^{(n)}\in\Phi'^{\odot n}$, $n\in\mathbb N$, and $\rho^{(0)}=1$. The  sequence of polynomials of binomial type, $(P^{(n)})_{n=0}^\infty$, with generating function \eqref{Eq23} is called the {\it basic sequence for the Sheffer sequence} $(S^{(n)})_{n=0}^\infty$.

A Sheffer sequence $(S^{(n)})_{n=0}^\infty$ with generating function \eqref{tyqde6iq} in which $A(\xi)=\xi$ is called an {\it Appell sequence on $\Phi'$}.  Thus,  in this case, the sequence $(S^{(n)})_{n=0}^\infty$ has the  generating function
\begin{equation}\label{tedf66iq}
\sum_{n=0}^\infty\frac{1}{n!}\langle S^{(n)}(\omega),\xi^{\otimes n}\rangle
=\frac{\exp[\langle\omega,\xi\rangle]}{\rho(\xi)},\end{equation}
where $\rho(\xi)$ is given by \eqref{edfuyr7i}. 
Equivalently, an Appell sequence is a Sheffer sequence for which the basic sequence is  $(\omega^{\otimes n})_{n=0}^\infty$.

\subsection{Sheffer homeomorphisms}\label{qwrsdreddw}

Let $(S^{(n)})_{n=0}^\infty$ be a Sheffer sequence on $\Phi'$. Consider the  bijective linear mapping\linebreak  $\mathfrak S:\mathcal P(\Phi')\to\mathcal P(\Phi')$ given by the formula \eqref{xreas4q4aas} in which $P^{(n)}$ is replaced by $S^{(n)}$. 
 Then $\mathfrak S$ is called the {\it Sheffer operator corresponding to the Sheffer sequence $(S^{(n)})_{n=0}^\infty$}. In the special case where $(S^{(n)})_{n=0}^\infty$ is of  binomial type, we will denote the corresponding Sheffer operator by $\mathfrak U$ and call it an {\it umbral operator}.   

%\begin{remark}
%For  a more  general definition of  Sheffer and umbral operators acting on  $\mathcal P(\Phi')$ and properties of such operators, we refer to \cite{FKLO2}. 
%\end{remark}

The following theorem is the main result of the paper.

\begin{theorem}\label{sew56uw} Let $(S^{(n)})_{n=0}^\infty$ be a Sheffer sequence with generating function \eqref{tyqde6iq}. Assume that $\rho$ is holomorphic on a neighborhood of zero in $\Phi$. Let $B(\xi)$ denote the compositional inverse of $A(\xi)$, and assume that  the formal power series $A(\xi)$ and $B(\xi)$ are 
quasi-holomorphic on a neighborhood of zero in $\Phi$ (see Definition~\ref{64834w7}).  Let $\alpha\in[0,1]$. Then the corresponding Sheffer operator $\mathfrak S$  extends
by continuity to a linear self-homeomorphism of $\mathcal E^{\alpha}_{\mathrm{min}}(\Phi')$.
\end{theorem}

\begin{remark}\label{cxtew55u}
Let $O$ be an open neighborhood of zero in $\mathbb C^n$ and let $f:O\to\mathbb C^n$ ($n\in\mathbb N$) be a holomorphic function  such that $f(0)=0$ and the differential (the Jacobian matrix) of $f$ at 0 is the identity operator on $\mathbb C^n$. Then, as we know from  complex multivariate analysis,  the function  $f$ is invertible on a neighborhood of zero and its inverse function, $f^{-1}$, is holomorphic. Hence, it would be natural to expect that if $A(\xi)=\sum_{k=1}^\infty A_k\xi^{\otimes k}$ with $A_1=\mathbf 1$ is quasi-holomorphic on a neighborhood of zero in $\Phi$, then so is its compositional inverse $B(\xi)$.  It is an open problem whether this is indeed the case. 
\end{remark}

\begin{remark}\label{cr5w7k854ea}
In view of Theorem~\ref{tctfdr6tqde6} and the definition  of the projective limit topology, a linear operator $\mathfrak A:\mathcal E^{\alpha}_{\mathrm{min}}(\Phi')\to \mathcal E^{\alpha}_{\mathrm{min}}(\Phi') $ is continuous if and only if for any $\tau\in T$ and $l\in\mathbb N_0$ there exist $\tau'\in T$ and $l'\in\mathbb N_0$ such that $\mathfrak A$ extends to $\mathfrak A\in\mathcal L\big(\mathbf E^\alpha_{l'}(\mathcal H_{-\tau'}),\mathbf E^\alpha_{l}(\mathcal H_{-\tau})\big)$.
\end{remark}

Before proving Theorem \ref{sew56uw}, let us first formulate  an immediate corollary.

\begin{corollary}\label{tqde6} Let $(S^{(n)})_{n=0}^\infty$ be a Sheffer sequence as in Theorem \ref{sew56uw}.

 {\rm (i)} Let $\alpha\in[0,1]$. Each function $f\in \mathcal E^{\alpha}_{\mathrm{min}}(\Phi')$ admits a unique representation as 
\begin{equation}\label{d6esd6q}
f(\omega)=\sum_{n=0}^\infty\langle S^{(n)}(\omega),\varphi^{(n)}\rangle,\end{equation}
where $\varphi^{(n)}\in\Phi^{\odot n}$, $n\in\mathbb N_0$, and the series on the right-hand side of \eqref{d6esd6q} converges in $\mathcal E^{\alpha}_{\mathrm{min}}(\Phi')$.

{\rm (ii)} Let $\alpha\in(0,1]$ and 
let $f\in \mathcal E^{\alpha}_{\mathrm{min}}(\Phi')$ be of the form \eqref{d6esd6q}. For each $\tau\in T$ and $l\in \mathbb N_0$, define
\begin{equation}\label{ctes5w3wq}
 \vertiii{f}_{\tau,l,\alpha}:= \sum_{n=0}^\infty (n!)^{1/\alpha}\, 2^{ln}\|\varphi^{(n)}\|_\tau<\infty.
\end{equation}
Then $ \vertiii{\cdot}_{\tau,l,\alpha}$ determines a norm on $\mathcal E^{\alpha}_{\mathrm{min}}(\Phi')$, and we denote by $\mathbb E_{l,\tau}^\alpha$ the Banach space obtained as the completion of $\mathcal E^{\alpha}_{\mathrm{min}}(\Phi')$ in this norm. Then
$$\mathcal E^{\alpha}_{\mathrm{min}}(\Phi')=\operatornamewithlimits{proj\,lim}_{\tau\in T,\,l\in  \N_0}\mathbb E^{\alpha}_{l,\tau}.$$
\end{corollary}

\begin{remark}
Note that we do not state that, for given $\alpha\in(0,1]$, $\tau\in T$, and $l\in\mathbb N_0$, $\mathbb E^{\alpha}_{l,\tau}$ consists of entire functions on $\mathcal H_{-\tau}$. Nevertheless, for a given $\tau\in T$, one can find $\tau'\in T$ and $l'\in\mathbb N_0$ such that $\mathbb E^{\alpha}_{l',\tau'}$ consists of entire functions on $\mathcal H_{-\tau}$. 
\end{remark}

\begin{proof}[Proof of Theorem \ref{sew56uw}] In view of definition \eqref{tye64w3w}, it is sufficient to prove the result for $\alpha\in(0,1]$. 
We divide the proof of this case into several steps.

{\it Step 1}. We will first prove the result for the umbral operator
$\mathfrak U$ corresponding to a  sequence of polynomials of binomial type, 
$(P^{(n)})_{n=0}^\infty$. By \eqref{Eq23} and \eqref{ft7r}, we have
\begin{align}
\sum_{n=0}^\infty \frac1{n!}\langle P^{(n)}(\omega),\xi^{\otimes n}\rangle&=\sum_{m=0}^\infty \frac1{m!}\left(\sum_{k=1}^\infty \langle A_k^*\omega,\xi^{\otimes k}\rangle\right)^m\notag\\
&=1+\sum_{n=1}^\infty
\sum_{m=1}^n\frac1{m!}\sum_{\substack{k_1,\dots,k_m\in\mathbb N\\k_1+\dots+k_m=n}}\langle (A_{k_1}^*\omega)\odot\dotsm\odot(A_{k_m}^*\omega),\xi^{\otimes n}\rangle.\label{drwqw22}
\end{align}
Hence, for $n\in\mathbb N$,
\begin{equation}\label{ywdex}
P^{(n)}(\omega)=n!
\sum_{m=1}^n\frac1{m!}\sum_{\substack{k_1,\dots,k_m\in\mathbb N\\k_1+\dots+k_m=n}}(A_{k_1}^*\omega)\odot\dotsm\odot(A_{k_m}^*\omega).
\end{equation}
Let $p(\omega)=\sum_{n=0}^N\langle\omega^{\otimes n},\varphi^{(n)}\rangle\in\mathcal P(\Phi')$. Then by \eqref{ywdex} and the definition of $\mathfrak U$,
\begin{align}
&\mathfrak Up(\omega)=\sum_{n=0}^N \langle P^{(n)}(\omega),\varphi^{(n)}\rangle\notag\\
&=\varphi^{(0)}+\sum_{n=1}^N n!
\sum_{m=1}^n\frac1{m!}\sum_{\substack{k_1,\dots,k_m\in\mathbb N\\k_1+\dots+k_m=n}}\langle\omega^{\otimes m},(A_{k_1}\otimes\dots\otimes A_{k_m})\varphi^{(n)}\rangle\notag\\
&=\varphi^{(0)}+\sum_{m=1}^N \bigg\langle \omega^{\otimes m},\frac1{m!}\sum_{n=m}^N n! \sum_{\substack{k_1,\dots,k_m\in\mathbb N\\k_1+\dots+k_m=n}}(A_{k_1}\otimes\dots\otimes A_{k_m})\varphi^{(n)}\bigg\rangle\notag\\
&=\varphi^{(0)}+\sum_{m=1}^N\bigg\langle \omega^{\otimes m},\frac1{m!}\sum_{k_1=1}^\infty\dotsm\sum_{k_m=1}^\infty(k_1+\dots+k_m)! \,(A_{k_1}\otimes\dots\otimes A_{k_m})\varphi^{(k_1+\dots+k_m)}\bigg\rangle,\label{ftde7r5i67}
\end{align}
where we set $\varphi^{(n)}=0$ for $n\ge N+1$.
Fix $\tau\in T$ and $l\in\mathbb N_0$. Then,  by  \eqref{yufwrq76rf} and \eqref{ftde7r5i67},
\begin{multline}\label{xrsw5uw}
\|\mathfrak Up\|_{\tau,l,\alpha}\le|\varphi^{(0)}|+\sum_{m=1}^N 2^{lm} (m!)^{\frac1\alpha-1}\\
\times  \sum_{k_1=1}^\infty\dotsm\sum_{k_m=1}^\infty(k_1+\dots+k_m)! \,
\|(A_{k_1}\otimes\dots\otimes A_{k_m})\varphi^{(k_1+\dots+k_m)}\|_\tau.
\end{multline}
By (QH), 
 there exist $\tau'\in T$ and a constant $C_5\ge0$ such that, for all $k\in\mathbb N$, $\|A_k\|_{\mathcal L(\mathcal H_{\tau'}^{\odot k},\mathcal H_\tau)}\le C_5^k$.
Hence, by \eqref{xrsw5uw},
\begin{equation}\label{xrtesa4qw}
\|\mathfrak Up\|_{\tau,l,\alpha}\le|\varphi^{(0)}|+\sum_{m=1}^N 2^{lm} (m!)^{\frac1\alpha-1}\sum_{k_1=1}^\infty\dotsm\sum_{k_m=1}^\infty(k_1+\dots+k_m)! \, C_5^{k_1+\dots+k_m}\|\varphi^{(k_1+\dots+k_m)}\|_{\tau'}.
\end{equation}
We evidently have, for each $l'\in\mathbb N_0$,
\begin{equation}\label{5w3g}
\|\varphi^{(k)}\|_{\tau'}\le 2^{-l'k}(k!)^{-1/\alpha}\|p\|_{\tau',l',\alpha}\,,\quad k\in\mathbb N_0.
\end{equation}
By \eqref{xrtesa4qw} and \eqref{5w3g},
\begin{align}
\|\mathfrak Up\|_{\tau,l,\alpha}&\le|\varphi^{(0)}|+\sum_{m=1}^N 2^{lm} (m!)^{\frac1\alpha-1}\sum_{k_1=1}^\infty\dotsm\sum_{k_m=1}^\infty\big((k_1+\dots+k_m)!\big)^{1-\frac1\alpha}\notag\\
&\quad\times \left(\frac{C_5}{2^{l'}}\right)^{k_1+\dots+k_m}\|p\|_{\tau',l',\alpha}.\label{dgyder67}
\end{align}
Since $k_1+\dots+k_m\ge m$ and $1-\frac1\alpha\le 0$, formula \eqref{dgyder67} implies
\begin{align}
\|\mathfrak Up\|_{\tau,l,\alpha}&\le |\varphi^{(0)}|+\sum_{m=1}^N 2^{lm} \left(\sum_{k=1}^\infty \left(\frac{C_5}{2^{l'}}\right)^k\right)^m \|p\|_{\tau',l',\alpha}\notag\\
&=|\varphi^{(0)}|+\sum_{m=1}^N 2^{lm}\left(\frac{C_5}{2^{l'}-C_5}\right)^m
\|p\|_{\tau',l',\alpha}\notag\\
&\le \|p\|_{\tau',l',\alpha}\sum_{m=0}^\infty \left(\frac{2^lC_5}{2^{l'}-C_5}\right)^m\notag\\
&=\|p\|_{\tau',l',\alpha}\left(1-\frac{2^lC_5}{2^{l'}-C_5}\right)^{-1}\label{y6e6i4}
\end{align}
for $l'\in\mathbb N_0$ satisfying $2^{l'}>C_5(1+2^l)$. Hence, by Theorem~\ref{tctfdr6tqde6}, $\mathfrak U$ extends  by continuity to $\mathfrak U\in\mathcal L(\mathcal E^{\alpha}_{\mathrm{min}}(\Phi'))$.

{\it Step 2}. By using  formula \eqref{gu7r} and analogously to \eqref{drwqw22}, \eqref{ywdex}, we see that 
$$ \omega^{\otimes n}=n!
\sum_{m=1}^n\frac1{m!}\sum_{\substack{k_1,\dots,k_m\in\mathbb N\\k_1+\dots+k_m=n}}(B_{k_1}^*\odot\dotsm \odot B_{k_m}^*)P^{(m)}(\omega).
$$
Hence, analogously to \eqref{ftde7r5i67}, we find, for $p(\omega)=\sum_{n=0}^N\langle\omega^{\otimes n},\varphi^{(n)}\rangle\in\mathcal P(\Phi')$,
\begin{align}
&p(\omega)=\varphi^{(0)}\notag\\
&\text{}+\sum_{m=1}^N\bigg\langle P^{(m)}(\omega),\frac1{m!}\sum_{k_1=1}^\infty\dotsm\sum_{k_m=1}^\infty(k_1+\dots+k_m)! \,(B_{k_1}\otimes\dots\otimes B_{k_m})\varphi^{(k_1+\dots+k_m)}\bigg\rangle.\notag
\end{align}
This implies
\begin{align}
&\mathfrak U^{-1}p(\omega)=\varphi^{(0)}\notag\\
&\text{}+\sum_{m=1}^N\bigg\langle \omega^{\otimes m},\frac1{m!}\sum_{k_1=1}^\infty\dotsm\sum_{k_m=1}^\infty(k_1+\dots+k_m)! \,(B_{k_1}\otimes\dots\otimes B_{k_m})\varphi^{(k_1+\dots+k_m)}\bigg\rangle.\notag
\end{align}
Analogously to \eqref{xrsw5uw}--\eqref{y6e6i4}, we now conclude that $\mathfrak U^{-1}$ can be extended by continuity to an operator $\mathfrak U'\in\mathcal L(\mathcal E^{\alpha}_{\mathrm{min}}(\Phi'))$.
Since $\mathfrak U'=\mathfrak U^{-1}$ on $\mathcal P(\Phi')$, we get
$\mathfrak U\mathfrak U'p=\mathfrak U'\mathfrak Up=p$
for each $p\in \mathcal P(\Phi')$.  By continuity, this implies $\mathfrak U\mathfrak U'f=\mathfrak U'\mathfrak Uf=f$ for all $f\in \mathcal E^{\alpha}_{\mathrm{min}}(\Phi')$. Therefore, $\mathfrak U\in\mathcal L(\mathcal E^{\alpha}_{\mathrm{min}}(\Phi'))$ is a bijective mapping and $\mathfrak U^{-1}=\mathfrak U'\in \mathcal L(\mathcal E^{\alpha}_{\mathrm{min}}(\Phi'))$. Thus, the statement of the theorem (hence also the statement of Corollary~\ref{tqde6}) are proved in the case of a sequence of polynomials of binomial type.

{\it Step 3}.
Let now $(S^{(n)})_{n=0}^\infty$ be a Sheffer sequence and let $(P^{(n)})_{n=0}^\infty$ be its basic sequence.  We will use the  lemma below, which  follows  from \cite{FKLO}, Theorem~6.2, the~statement (SS4) with $\omega=0$, and Corollary~6.6.  

\begin{lemma}  Let  sequences of polynomials $(P^{(n)})_{n=0}^\infty$ and $(S^{(n)})_{n=0}^\infty$  have generating functions \eqref{Eq23} and \eqref{tyqde6iq}, respectively. Then, for each $n\in\mathbb N$,
\begin{align}
S^{(n)}(\omega)&=\sum_{k=0}^n \frac{n!}{(n-k)!}\, \theta^{(k)}\odot P^{(n-k)}(\omega),\label{drs5es5w3w}\\
P^{(n)}(\omega)&= \sum_{k=0}^n\frac{n!}{(n-k)!}\,\varkappa^{(k)}\odot S^{(n-k)}(\omega),\label{te6e76i475}
\end{align}
where $\theta^{(k)},\varkappa^{(k)}\in\Phi'^{\odot k}$, $k\in\mathbb N_0$, are given through the formulas 
\begin{align}
\frac1{\rho(A(\xi))}&=\sum_{k=0}^\infty
\langle\theta^{(k)},\xi^{\otimes k}\rangle,\label{xea4tq5}\\
\rho(A(\xi))&=\sum_{k=0}^\infty
\langle\varkappa^{(k)},\xi^{\otimes k}\rangle.\label{rtdssaeew5q7}
\end{align}
\end{lemma}

By Corollary \ref{e6r656e3}, there exist $\tau\in T$ and $C_6\ge0$ such that the function $\rho$ extends by continuity to a holomorphic function $\rho:U_{\tau}\to\mathbb C$, where 
$$U_{\tau}:=\{\xi\in\mathcal H_{\tau}\mid \|\xi\|_{\tau}<C_6^{-1}\}$$
 is an open neighborhood of zero in $\mathcal H_{\tau}$. Furthermore, by the definition of a quasi-holomorphic formal power series, there exists an open neighborhood of zero in $\Phi$, denoted by $V$, such that the formal power series $A(\xi)$
 determines a holomorphic function $A:V\to\mathcal H_{\tau}$. Without loss of generality, we may assume that $A(V)\subset U_\tau$. 
 Hence, the function $\rho(A(\cdot)):V\to\mathbb C$ is holomorphic. Since $A(0)=0$ and $\rho(0)=1$, this also implies that $1/\rho(A(\cdot))$ is holomorphic on a neighborhood of zero.

We define a bijective linear mapping $\widetilde{ \mathfrak S}:\mathcal P(\Phi')\to \mathcal P(\Phi')$ by 
$$\widetilde {\mathfrak S} \langle P^{(n)}(\omega),\varphi^{(n)}\rangle=\langle S^{(n)}(\omega),\varphi^{(n)}\rangle\quad \varphi^{(n)}\in\Phi^{\odot n},\ n\in\mathbb N_0.$$

For $\varphi^{(n)}\in\Phi^{\odot n}$ and $\Psi^{(k)}\in\Phi'^{\odot k}$ with $k<n$, we denote by $\langle \Psi^{(k)},\varphi^{(n)}\rangle$ the element  of $\Phi^{\odot(n-k)}$ that satisfies
$$ \big\langle \Gamma^{(n-k)},\langle \Psi^{(k)},\varphi^{(n)}\rangle\big\rangle=\langle  \Gamma^{(n-k)}\odot \Psi^{(k)},\varphi^{(n)}\rangle\quad\text{for all }\Gamma^{(n-k)}\in\Phi'^{\odot(n-k)}. $$

Let 
$p(\omega)=\sum_{n=0}^N \langle P^{(n)}(\omega),\varphi^{(n)}\rangle\in\mathcal P(\Phi')$. By \eqref{drs5es5w3w},  we get
\begin{align}
\widetilde{\mathfrak S}p(\omega)&=\varphi^{(0)}+\sum_{n=1}^N\sum_{k=0}^n\frac{n!}{k!}\,
\big\langle P^{(k)}(\omega), \langle\theta^{(n-k)},\varphi^{(n)}\rangle\big\rangle\notag\\
&=\sum_{n=0}^Nn!\,\langle\theta^{(n)},\varphi^{(n)}\rangle+\sum_{k=1}^N\bigg\langle P^{(k)}(\omega),\sum_{n=k}^N\frac{n!}{k!} \langle\theta^{(n-k)},\varphi^{(n)}\rangle\bigg\rangle.\label{tyrte6i4}
\end{align}
Let $ \vertiii{\cdot}_{\tau,l,\alpha}$ ($\tau\in  T$, $l\in \mathbb N_0$)
denote the norms on $\mathcal E^{\alpha}_{\mathrm{min}}(\Phi')$ 
as defined in Corollary~\ref{tqde6},~(ii), but associated with the 
sequence of polynomials $(P^{(n)})_{n=0}^\infty$. Let $\tau\in T$ be fixed. 
In view of \eqref{xea4tq5} and Corollary \ref{e6r656e3}, without loss of generality we may assume that, for some $C_7\ge1$, we have
$\|\theta^{(n)}\|_{-\tau}\le C_7^n$ for all  $n\in\mathbb N$. 
By \eqref{ctes5w3wq},
$$\|\varphi^{(n)}\|_\tau\le (n!)^{-1/\alpha}2^{-ln}\vertiii{p}_{\tau,l,\alpha},\quad n=0,1,\dots,N,\ l\in \mathbb N_0.$$
Hence, by \eqref{tyrte6i4}, we get for $l,l'\in\mathbb N_0$,
\begin{align}
&\vertiii{\widetilde{\mathfrak S}p}_{\tau,l,\alpha}=
\left|\sum_{n=0}^Nn!\,\langle\theta^{(n)},\varphi^{(n)}\rangle\right|+\sum_{k=1}^N (k!)^{1/\alpha}2^{lk}\left\|\sum_{n=k}^N \frac{n!}{k!}\, \langle\theta^{(n-k)},\varphi^{(n)}\rangle\right\|_\tau\notag\\
&\le \sum_{n=0}^N n!\, \|\theta^{(n)}\|_{-\tau}\|\varphi^{(n)}\|_\tau +\sum_{k=1}^N (k!)^{1/\alpha}2^{lk} \sum_{n=k}^N\frac{n!}{k!}\, \|\theta^{(n-k)}\|_{-\tau}\|\varphi^{(n)}\|_\tau\notag\\
&\le\left(\sum_{n=0}^N C_7^n  (n!)^{1-\frac1\alpha}2^{-l'n}
+\sum_{n=1}^N\sum_{k=1}^n (k!)^{\frac1\alpha-1}2^{lk}C_7^{n-k}(n!)^{1-\frac1\alpha}2^{-l'n}
\right)\vertiii{p}_{\tau,l',\alpha}\label{ctesu6w}\\
&\le \left(\sum_{n=0}^N C_7^n  2^{-l'n}
+ \sum_{n=1}^N\sum_{k=1}^n 
\left(\frac{k!}{n!}\right)^{\frac1\alpha-1}2^{lk} C_7^{n-k}2^{-l'n}
\right)\vertiii{p}_{\tau,l',\alpha}\notag\\
&\le C_8\vertiii{p}_{\tau,l',\alpha},\notag
\end{align}
where
$$C_8:=\sum_{n=0}^\infty\left(\frac{C_7}{2^{l'}}\right)^n+\sum_{n=0}^\infty n\left(\frac{2^lC_7}{2^{l'}}\right)^n<\infty$$
for sufficiently large $l'\in\mathbb N$. Hence,  $\widetilde {\mathfrak S}$  extends by continuity to $\widetilde {\mathfrak S}\in\mathcal L(\mathcal E^{\alpha}_{\mathrm{min}}(\Phi'))$.

{\it Step 4}. Finally, using formulas \eqref{te6e76i475} and \eqref{rtdssaeew5q7}, analogously to Steps 2 and 3, we prove that $\widetilde {\mathfrak S}^{-1}$  extends by continuity to a linear continuous operator on  $\mathcal E^{\alpha}_{\mathrm{min}}(\Phi')$. Therefore, $\widetilde {\mathfrak S}\in \mathcal L(\mathcal E^{\alpha}_{\mathrm{min}}(\Phi'))$ is a self-homeomorphism, hence so is $\mathfrak S=\widetilde {\mathfrak S}\mathfrak U$. 
\end{proof}

Now a natural question arises as to whether it is possible to extend the result of Theorem \ref{sew56uw} to $\alpha>1$. The following  proposition, which deals with the one-dimensional case only, immediately implies that this not always possible.

\begin{proposition} \label{6ew6u47eies}
{\rm (i)}  Let $(s_n)_{n=0}^\infty$ be a  Sheffer sequence on $\mathbb C$, and  let $\mathfrak S$ be the corresponding Sheffer operator. Assume that, for some $\alpha>1$, $\mathfrak S$  extends by continuity to $\mathfrak S\in\mathcal L(\mathcal E^{\alpha}_{\mathrm{min}}(\mathbb C))$. 
  Denote by 
\begin{equation}\label{yd6ed6}
G(u,z):=\sum_{n=0}^\infty \frac{u^n}{n!}\,s_n(z)\end{equation}
the  generating function of $(s_n)_{n=0}^\infty$. Then, for all  $(u,z)\in\mathbb C^2$, the series on the right-hand side of \eqref{yd6ed6} converges in $\mathbb C$, and $G:\mathbb C^2\to\mathbb C$ is an entire function.

{\rm (ii)} Let  $(p_n)_{n=0}^\infty$ be a   sequence of polynomials on $\mathbb C$ of binomial type, and  let $\mathfrak U$ be the corresponding umbral operator. Then, for each $\alpha>1$, $\mathfrak U$ cannot be extended by continuity to $\mathfrak U\in\mathcal L(\mathcal E^{\alpha}_{\mathrm{min}}(\mathbb C))$, unless $p_n(z)=z^n$ (in which case $\mathfrak U$ is the identity operator).
\end{proposition}

\begin{proof} (i)
Fix $u\in\mathbb C$ and define $\exp_u(z):=e^{uz}=\sum_{n=0}^\infty\frac{u^n}{n!}\,z^n$. Obviously, $\exp_u\in \mathcal E^{\alpha}_{\mathrm{min}}(\mathbb C)$. Hence,
\begin{equation}\label{d5w6u43}
(\mathfrak S\exp_u)(z)=\sum_{n=0}^\infty  \frac{u^n}{n!}\,s_n(z)=G(u,z),\end{equation}
and the series converges in $\mathcal E^{\alpha}_{\mathrm{min}}(\mathbb C)$, in particular, it converges point-wise. Note that  $G(u,\cdot)\in \mathcal E^{\alpha}_{\mathrm{min}}(\mathbb C)$, in particular, $G(u,\cdot)$ is an entire function on $\mathbb C$. Furthermore, by \eqref{d5w6u43}, for each fixed $z\in\mathbb C$, $G(\cdot,z)$ is an entire function on $\mathbb C$. It is easy to see that  the function $G:\mathbb C^2\to\mathbb C$ is continuous, hence it is entire. 

(ii) Assume there exists $\alpha>1$ such that $\mathfrak U$  extends by continuity to $\mathfrak U\in\mathcal L(\mathcal E^{\alpha}_{\mathrm{min}}(\mathbb C))$. By (i), the corresponding generating function $G(u,z)=\sum_{n=0}^\infty  \frac{u^n}{n!}\,p_n(z)$ is entire on $\mathbb C^2$. Recall that 
$G(u,z)=\exp[za(u)]$. Since
$a(u)=\frac{\partial}{\partial z}\Big|_{z=0}G(u,z)$,
$a(u)$ is an entire function on $\mathbb C$. By Remark~\ref{cr5w7k854ea} and the continuity of $\mathfrak U$, there exists $l\in\mathbb N_0$ such that $\mathfrak U$ extends to $\mathfrak U\in\mathcal L(\mathbf E_l^\alpha(\mathbb C), \mathbf E_0^\alpha(\mathbb C))$. Hence, there exists $C_9>0$ such that $\|G(u,\cdot)\|_{0,\alpha}\le C_9\|\exp_u\|_{l,\alpha} $ for all $u\in\mathbb C$ (we obviously dropped $\tau$ from the notation $\|\cdot\|_{\tau,l,\alpha}$). Therefore, 
\begin{equation}\label{q22evgfgf}
\sum_{n=0}^\infty(n!)^{\frac1\alpha-1}|a(u)|^n\le C_9\sum_{n=0}^\infty(n!)^{\frac1\alpha-1}\,2^{ln}|u|^n.
\end{equation}
For $s\in[0,\infty)$, define $\phi(s):=\sum_{n=0}^\infty(n!)^{\frac1\alpha-1}s^n$. Then, inequality \eqref{q22evgfgf} becomes 
\begin{equation}\label{tw5u36i7}
\phi(|a(u)|)\le C_9\phi(2^l|u|).\end{equation}
As easily seen, there exists a constant $C_{10}>0$ such that
%\begin{equation}\label{rtw5}
$C_9\phi(2^ls)\le \phi(C_{10}s)$ for all $s\ge1$.
Hence, \eqref{tw5u36i7} implies
\begin{equation}\label{rtw52}
\phi(|a(u)|) \le \phi(C_{10}|u|)\quad\text{for }|u|\ge1.\end{equation}
Since the function $\phi:[0,\infty)\to\mathbb R$ is monotone increasing, formula \eqref{rtw52} implies that
%\begin{equation}\label{rseu45eszz}
$|a(u)|\le C_{10}|u|$ for $|u|\ge1$. Because $a(u)$ is an entire function, $a(0)=0$ and $a'(0)=1$, this implies that $a(u)=u$.
\end{proof}

In the case of an Appell sequence, we will now find a sufficient condition for the corresponding Sheffer operator $\mathfrak S$ to extend by continuity to a linear 
self-homeomorphism of  $\mathcal E^{\alpha}_{\mathrm{min}}(\Phi')$ for $\alpha>1$.  

\begin{theorem}\label{ft7r8o5}
Let $(S^{(n)})_{n=0}^\infty$ be an Appell sequence with generating function \eqref{tedf66iq}, and let 
$\beta\in(1,\infty)$.  Assume that there exist $ \tau\in T$ and a constant $C_{11}\ge1$ such that
\begin{equation}\label{cytd7i56e}
\|\theta^{(n)}\|_{-\tau}\le C_{11}^{n}(n!)^{\frac1\beta-1},\quad \|\rho^{(n)}\|_{- \tau}\le C_{11}^{n}(n!)^{\frac1\beta-1},\quad n\in\mathbb N,
\end{equation}
where $\theta^{(n)}$ and $\rho^{(n)}$ are given by \eqref{xea4tq5} (with $A(\xi)=\xi$) and  \eqref{edfuyr7i}, respectively. Then, for each  $\alpha\in[0,\beta]$, the corresponding Sheffer operator $\mathfrak S$ extends
by continuity to a linear self-homeomorphism of $\mathcal E^{\alpha}_{\mathrm{min}}(\Phi')$. 
\end{theorem}

\begin{remark}
Note that, in the limiting case $\beta=1$, the estimates \eqref{cytd7i56e} become $\|\theta^{(n)}\|_{-\tau}\le C_{11}^{n}$ and $\|\rho^{(n)}\|_{-\tau}\le C_{11}^{n}$, which means that the function $\rho(\cdot)$ is holomorphic on a neighborhood of zero in $\Phi$.
\end{remark}

\begin{proof}[Proof of Theorem \ref{ft7r8o5}] 
Note that, when $\beta$ is increasing, the right-hand side of the inequalities \eqref{cytd7i56e} is decreasing. Therefore, it is sufficient to prove the statement of the theorem only for the space $\mathcal E^{\alpha}_{\mathrm{min}}(\Phi')$ when $\alpha=\beta$. To this end, we just need to slightly modify Steps 3 and 4 of the proof of Theorem \ref{sew56uw}.
We will use the notation from that proof. So, analogously to \eqref{ctesu6w}, we obtain 
\begin{align}
\|\mathfrak Sp\|_{\tau,l,\alpha}\notag&\le\bigg(\sum_{n=0}^N C_{11}^n(n!)^{1+\frac1\alpha-1-\frac1\alpha}  2^{-l'n}\notag\\
&\quad
+\sum_{n=1}^N\sum_{k=1}^n (k!)^{\frac1\alpha-1}2^{lk}C_{11}^{n-k}
\big((n-k)!\big)^{\frac1\alpha-1}
(n!)^{1-\frac1\alpha}2^{-l'n}\bigg)\|p\|_{\tau,l',\alpha}\notag\\
&\le\left(\sum_{n=0}^N C_{11}^n  2^{-l'n}+\sum_{n=1}^N \bigg(\frac{2^lC_{11}}{2^{l'}}\bigg)^n\sum_{k=1}^n\binom nk^{1-\frac1\alpha}\right)\|p\|_{\tau,l',\alpha}\notag\\
&\le\left(\sum_{n=0}^N C_{11}^n  2^{-l'n}+\sum_{n=1}^N \bigg(\frac{2^lC_{11}}{2^{l'}}\bigg)^n \sum_{k=1}^n\binom nk\right)\|p\|_{\tau,l',\alpha}\notag\\
&\le\left(\sum_{n=0}^N C_{11}^n  2^{-l'n}+\sum_{n=1}^N \bigg(\frac{2^{l+1}C_{11}}{2^{l'}}\bigg)^n \right)\|p\|_{\tau,l',\alpha}\notag\\
&=C_{12}\|p\|_{\tau,l',\alpha},\notag
\end{align}
where $C_{12}<\infty$ for a sufficiently large $l'\in\mathbb N_0$. Hence,  $\mathfrak S$ extends by continuity to $\mathfrak S\in\mathcal L(\mathcal E^{\alpha}_{\mathrm{min}}(\Phi'))$. The rest of the proof is similar to Step 4 of the proof of Theorem~\ref{sew56uw}.
\end{proof}

Let us now apply our results to the finite-dimensional setting. Let $d\in\mathbb N$ and choose
$\Phi=\mathcal H_0=\Phi'=\mathbb C^d$.
In particular, the index set $T$ contains only one element.  Note also the result from the multivariate complex analysis that is recalled in Remark~\ref{cxtew55u}. Hence, we get the following  corollary.

\begin{corollary}\label{6u14we145crsd}
 {\rm (i)} Let $(S^{(n)})_{n=0}^\infty$ be a Sheffer sequence on $\mathbb C^d$ with the generating function~\eqref{tyqde6iq}. 
Assume that there exists $U$, an open neighborhood of zero in $\mathbb C^d$, such that the functions $A:U\to\mathbb C^d$ and $\rho:U\to\mathbb C$ are holomorphic. Then, for each $\alpha\in[0,1]$, the corresponding Sheffer operator $\mathfrak S: \mathcal P(\mathbb C^d)\to \mathcal P(\mathbb C^d)$ extends
by continuity to a linear self-homeomorphism of $\mathcal E^{\alpha}_{\mathrm{min}}(\mathbb C^d)$.

{\rm (ii)} Let $(S^{(n)})_{n=0}^\infty$ be an Appell sequence on $\mathbb C^d$ with generating function \eqref{tedf66iq}. Assume that there exist $\beta\in(1,\infty)$ and a constant $C_{13}\ge1$ such that 
\begin{equation}\label{twruqtddfrq6}
\|\theta^{(n)}\|\le C_{13}^{n}(n!)^{\frac1\beta-1},\quad \|\rho^{(n)}\|\le C_{13}^{n}(n!)^{\frac1\beta-1},\quad n\in\mathbb N,
\end{equation}
where $\theta^{(n)}$ and $\rho^{(n)}$ are given by \eqref{xea4tq5} (with $A(\xi)=\xi$) and  \eqref{edfuyr7i}, respectively, and $\|\cdot\|$ denotes the norm on $\mathbb C^d$. Then, for each $\alpha\in[0,\beta]$, the corresponding Sheffer operator $\mathfrak S$ extends
by continuity to a linear self-homeomorphism of $\mathcal E^{\alpha}_{\mathrm{min}}(\mathbb C^d)$.

\end{corollary}

\section{Examples}\label{a43wq7}

We will now consider some examples of application of our results. 

\subsection{Hermite polynomials on $\mathbb C^d$}\label{cykreri64}

 Let $\mathcal C:\R^d\to\R^d$ be a symmetric, strictly positive linear operator. 
Let $\mu$ denote the centered Gaussian measure on $\R^d$ with correlation operator $\mathcal C$:
$$\int_{\R^d}e^{i\langle\omega,\xi\rangle}\,d\mu(\omega)=\exp\bigg[-\frac12\, \langle \mathcal C\xi,\xi\rangle\bigg],\quad\xi\in\R^d.$$
We extend $\mathcal C$ by linearity to a linear operator on $\mathbb C^d$, and consider the entire function $\rho:\mathbb C^d\to\mathbb C$ given by
\begin{equation}\label{dr6e8}
\rho(\xi)=\exp\bigg[\frac12\, \langle \mathcal C\xi,\xi\rangle\bigg].\end{equation}
Let  $(S^{(n)})_{n=0}^\infty$ be the corresponding Appell sequence on $\mathbb C^d$ with the generating function~\eqref{tedf66iq}.
Then $(S^{(n)})_{n=0}^\infty$ is a sequence of {\it multivariate Hermite polynomials}. In particular, for any $m,n\in\mathbb N_0$, $f^{(m)}\in(\mathbb R^d)^{\odot m}$, $g^{(n)}\in(\mathbb R^d)^{\odot n}$, one has
$$\int_{\R^d} \langle S^{(m)}(\omega),f^{(m)}\rangle \langle S^{(n)}(\omega),g^{(n)}\rangle\,d\mu(\omega)=\delta_{m,n}n!\,(\mathcal C^{\otimes n}f^{(n)},g^{(n)})_{(\R^d)^{\odot n}}\,,$$
where $\delta_{m,n}$ is the Kronecker symbol, see e.g.\ \cite[Chapter 2, Section 2.2]{BK}

Let $\tilde \Delta\in(\mathbb C^d)^{\otimes 2}$ be given by
$\langle\tilde\Delta,\xi_1\otimes \xi_2\rangle=\langle\mathcal C\xi_1,\xi_2\rangle$
for all $\xi_1,\xi_2\in\mathbb C^d$.
(In fact, $\tilde\Delta\in(\R^d)^{\otimes 2}\subset (\mathbb C^d)^{\otimes 2}$.) Let $\Delta\in(\mathbb C^d)^{\odot 2}$ denotes the symmetrization of $\tilde\Delta$. In particular, for all $\xi\in\mathbb C^d$,
$\langle\Delta,\xi^{\otimes 2}\rangle=\langle\mathcal C\xi,\xi\rangle
$. Then, by \eqref{dr6e8}, 
\begin{equation}\label{s1w71wrs}
\rho(\xi)=1+\sum_{k\in\mathbb N,\, \text{$k$ even}}
\frac1{(k/2)!\, 2^{k/2}}\langle \Delta^{\odot(k/2)},\xi^{\otimes k}\rangle,\quad \xi\in\mathbb C^d,
\end{equation}
and similarly
\begin{equation}\label{r6es}
\frac1{\rho(\xi)}=1+\sum_{k\in\mathbb N,\, \text{$k$ even}}
\frac{(-1)^{k/2}}{(k/2)!\, 2^{k/2}}\langle \Delta^{\odot(k/2)},\xi^{\otimes k}\rangle,\quad \xi\in\mathbb C^d.\end{equation}
It easily follows from \eqref{s1w71wrs} and \eqref{r6es} that the functions $\rho(\xi)$ and $1/\rho(\xi)$ satisfy condition \eqref{twruqtddfrq6} with $\beta=2$. Hence, by Corollary \ref{6u14we145crsd}, (ii), {\it  for each $\alpha\in[0,2]$, the corresponding Sheffer operator $\mathfrak S$ extends by continuity to a linear self-homeomorphism of $\mathcal E^{\alpha}_{\mathrm{min}}(\mathbb C^d)$}.

\subsection{Infinite dimensional Hermite polynomials}
The above considerations related to the  multivariate Hermite polynomials admit a straightforward generalization to the case of a general Gel'fand triple \eqref{fd6re7itfrde}. Let $\tau\in T$ and let $\mathcal C\in\mathcal L(\mathcal H_{\tau,\R},\mathcal H_{-\tau,\R})$ be such that, for each $\xi\in\mathcal H_{\tau,\R}$, $\xi\ne0$, we have $\langle \mathcal C\xi,\xi\rangle>0$. Let $\mu$ denote the Gaussian measure on $\Phi_{\R}'$ (equipped with the cylinder $\sigma$-algebra) with correlation operator $\mathcal C$:
$$\int_{\Phi_{\R}'}e^{i\langle\omega,\xi\rangle}\,d\mu(\omega)=\exp\bigg[-\frac12\, \langle \mathcal C\xi,\xi\rangle\bigg],\quad\xi\in\Phi_\R$$
(see e.g.\ \cite[Chapter 2, Section 1.9]{BK}). Similarly to Subsection \ref{cykreri64}, we define an entire function $\rho:\Phi\to\mathbb C$ by formula \eqref{dr6e8}. Let  $(S^{(n)})_{n=0}^\infty$ be the corresponding Appell sequence on $\Phi'$ with generating function \eqref{tedf66iq}.
Then $(S^{(n)})_{n=0}^\infty$ is a sequence of {\it Hermite polynomials on $\Phi'$}. In particular, for any $m,n\in\mathbb N_0$, $f^{(m)}\in\Phi_{\R}^{\odot m}$, $g^{(n)}\in\Phi_{\R}^{\odot n}$, one has 
$$\int_{\Phi_{\R}'} \langle S^{(m)}(\omega),f^{(m)}\rangle \langle S^{(n)}(\omega),g^{(n)}\rangle\,d\mu(\omega)=\delta_{m,n}n!\,\langle\mathcal C^{\otimes n}f^{(n)},g^{(n)}\rangle\,.$$

Consider the continuous bilinear form $\langle \mathcal C\xi_1,\xi_2\rangle$ on $\mathcal H_\tau^2$. Let $\tau'\in T$ be such that $\mathcal H_{\tau'}\subset\mathcal H_\tau$ and the embedding operator is of Hilbert--Schmidt class. Then, by the Kernel Theorem, there exists $\tilde \Delta\in\mathcal H_{-\tau'}^{\otimes 2}$ satisfying $\langle\tilde\Delta,\xi_1\otimes \xi_2\rangle=\langle\mathcal C\xi_1,\xi_2\rangle$ for all $\xi_1,\xi_2\in\mathcal H_{\tau'}$. Let $\Delta\in \mathcal H_{-\tau'}^{\odot 2}$ denote the symmetrization of $\tilde\Delta$. In particular, for all $\xi\in\Phi$, we have $\langle\Delta,\xi^{\otimes 2}\rangle=\langle\mathcal C\xi,\xi\rangle
$. Hence, analogously to Subsection~\ref{cykreri64}, we conclude from Theorem~\ref{ft7r8o5}, that, {\it for each $\alpha\in[0,2]$,  the corresponding Sheffer operator $\mathfrak S$ extends by continuity to a linear self-homeomorphism of $\mathcal E^{\alpha}_{\mathrm{min}}(\Phi')$}.

For $\alpha=2$, this result was shown in \cite[Chapter 2, Section 5.4]{BK} and, in the case of the standard Gaussian measure ($\mathcal C=\mathbf 1$), it is a fundamental result of (Gaussian) white noise analysis, see also \cite{Lee,KPS,MY}. It gives an internal description of the space of test functions whose dual space is called the {\it space of Hida distributions}. For $\alpha=1$ and $\mathcal C=\mathbf 1$, this result was shown in \cite[Theorem~9]{KSW}. In this case, it gives an inner description of the space of test functions whose dual space is nowadays called the {\it space of Kondratiev distributions}.

\subsection{Lifted Sheffer sequences}

Let $\mathcal H_{0,\R}=L^2(\mathbb R^d,dx)$, the real $L^2$-space on $\mathbb R^d$, so that $\mathcal H_0=L^2(\mathbb R^d\to\mathbb C,dx)$. By the nuclear space $\Phi$ we will now understand either $\mathcal D$, the space  of all complex-valued smooth functions on $\R^d$ with compact support, or $\mathcal S$, the Schwartz space of complex-valued rapidly decreasing  smooth functions on $\R^d$.

Let us now briefly recall the construction of lifting of a Sheffer sequence on $\mathbb C$ to a Sheffer sequence on $\Phi'=\text{$\mathcal D'$ or $\mathcal S'$}$, see \cite{FKLO} for details. For each $k\in\mathbb N$, we define an operator $\mathbb D_k\in\mathcal L(\Phi^{\odot k},\Phi)$ by
\begin{equation}\label{sew46u}
(\mathbb D_kf^{(k)})(x):=f^{(k)}(x,\dots,x),\quad f^{(k)}\in\Phi^{\odot n},\ x\in\mathbb R^d\end{equation}
($\mathbb D_1$ being the identity operator on  $\Phi$). In particular, for $\xi\in\Phi$, we get $\mathbb D_k\xi^{\otimes k}=\xi^k$.

Let $(s_n)_{n=0}^\infty$ be a Sheffer sequence on $\mathbb C$ with generating function \eqref{xer5w5y43}. We write the generating function in the form 
$$\sum_{n=0}^\infty \frac{u^n}{n!}\, s_n(z)=\exp\big[za(u)-c(a(u))\big],$$
where $c(u):= \log r(u)$. We write  $a(u)=\sum_{k=1}^\infty a_ku^k$ with $a_1=1$ and $c(u)=\sum_{k=1}^\infty c_ku^k$.

Let $(S^{(n)})_{n=0}^\infty$  be a Sheffer sequence on $\Phi'$. We say that $(S^{(n)})_{n=0}^\infty$ is the {\it lifting of the Sheffer sequence $(s_n)_{n=0}^\infty$} if the generating function of $(S^{(n)})_{n=0}^\infty$ is of the form 
$$\sum_{n=0}^\infty\frac{1}{n!}\langle S^{(n)}(\omega),\xi^{\otimes n}\rangle
=\exp\big[\langle\omega,A(\xi)\rangle-C(A(\xi))\big], $$
 where
\begin{align}
A(\xi)&=\sum_{k=1}^\infty a_k\mathbb D_k\xi^{\otimes k}=\sum_{k=1}^\infty a_k\xi^k=a(\xi),\label{dse5w5}\\
C(\xi)&=\sum_{k=1}^\infty c_k\int_{\mathbb R^d} (\mathbb D_k\xi^{\otimes k})(x)\,dx= \sum_{k=1}^\infty c_k \int_{\mathbb R^d} \xi^{k}(x)\,dx= \int_{\mathbb R^d} c(\xi(x))\,dx.\label{e5w5qy3wdd}
\end{align}

\begin{proposition}\label{xresw5yq}
Let $(s_n)_{n=0}^\infty$ be a Sheffer sequence on $\mathbb C$ with generating function \eqref{xer5w5y43}. Assume that functions $a(u)$ and $r(u)$ are holomorpohic on a neighborhood of zero in $\mathbb C$ (i.e., the assumption of Theorem~\ref{xtew4u} is satisfied). Let $(S^{(n)})_{n=0}^\infty$ be the lifting of the Sheffer sequence  $(s_n)_{n=0}^\infty$ to a Sheffer sequence on $\Phi'=\text{$\mathcal D'$ or $\mathcal S'$}$. Then $(S^{(n)})_{n=0}^\infty$ satisfies the assumptions of 
Theorem~\ref{sew56uw}, and hence, for each  $\alpha\in[0,1]$, the corresponding Sheffer operator $\mathfrak S$  extends
by continuity to a linear self-homeomorphism of $\mathcal E^{\alpha}_{\mathrm{min}}(\Phi')$.
\end{proposition}

\begin{proof} We will prove the proposition for $\Phi=\mathcal D$, the proof for $\mathcal S$ being similar. 
Let $T$ denote the set of all pairs $(m,\varphi)$ with $m\in\mathbb N_0$ and $\varphi\in C^\infty(\R^d)$, $\varphi(x)\ge1$ for all $x\in\R^d$. For each $\tau=(m,\varphi)\in T$, we denote by $\mathcal H_{\tau,\R}$ the Sobolev  space $W^{m,2}(\mathbb R^d, \varphi(x)\,dx)$,  and let $\mathcal H_\tau$ be its complexification. We have $\mathcal D=\projlim_{\tau\in T}\mathcal H_\tau$, see e.g.\ \cite[Chapter 14, Subsec.~4.3]{BSU2}.

 A straightforward generalization of \cite[Theorem~7.1]{BT} shows that, for $m\ge d+1$,  $\mathcal H_{\tau}$ is a Banach algebra under the pointwise multiplication of functions, i.e., for any $f,g\in \mathcal H_{\tau}$, we have $fg\in\mathcal H_{\tau}$, and furthermore there exists $C_{14}>0$ such that, for all  $f,g\in \mathcal H_{\tau}$, $\|fg\|_{\tau}\le C_{14}\|f\|_{\tau} \|g\|_{\tau}$.  This, in turn, implies that, for any $f_1,f_2,\dots,f_k\in\mathcal H_\tau$, we have $f_1f_2\dotsm f_k\in\mathcal H_\tau$ and 
 $$\|f_1f_2\dotsm f_k\|_\tau\le C_{14}^{k-1}\|f_1\|_\tau\|f_2\|_\tau\dotsm\|f_k\|_\tau. $$

Let $\tau'\in T$ be such that $\mathcal H_{\tau'}\subset\mathcal H_\tau$ and the embedding operator is of Hilbert--Schmidt class. Then, by the Kernel Theorem and its proof, see \cite[Chapter 1, Theorem 2.3]{BK},  the operator $\mathbb D_k$ given by \eqref{sew46u} extends by continuity to a bounded linear operator from $\mathcal H_{\tau'}^{\odot k}$ into $\mathcal H_\tau$, and furthermore there exists a constant $C_{15}>0$ such that
\begin{equation}\label{rqdtqfdta}
\|\mathbb D_k\|_{\mathcal L(\mathcal H_{\tau'}^{\odot k},\mathcal H_\tau)}\le C_{15}^k,\quad k\in\mathbb N.\end{equation}

 The function $a(u)=\sum_{k=1}^\infty a_ku^k$ is holomorphic on a neighborhood of zero in $\mathbb C$. Hence, there exists a constant $C_{16}>0$ such that $|a_k|\le C_{16}^k$ for all $k\in\mathbb N$. Therefore, by \eqref{rqdtqfdta}, for each $\tau\in T$, there exist $\tau'\in T$ and a constant $C_{17}>0$ such that 
$$\|a_k\mathbb D_k\|_{\mathcal L(\mathcal H_{\tau'}^{\odot k},\mathcal H_\tau)}\le C_{17}^k,\quad k\in\mathbb N.$$
By \eqref{dse5w5}, this implies that the formal power series $A(\xi)$ is quasi-holomorphic on a neighborhood of zero. 

Let $b(u)=\sum_{k=1}^\infty b_ku^k$ be the compositional inverse of $a(u)$. 
Then 
$$B(\xi)=\sum_{k=1}^\infty b_k\mathbb D_k\xi^{\otimes k}=\sum_{k=1}^\infty b_k\xi^k=b(\xi)$$
 is the compositional inverse of $A(\xi)$.
Since $b(u)$ is holomorphic on a neighborhood of zero, we analogously conclude that $B(\xi)$ is quasi-holomorphic on a neighborhood of zero.

Finally, since $r(u)$ is holomorphic on a neighborhood of zero and $r(0)=1$, the function $c(u)$ is holomorphic on a neighborhood of zero. Therefore, the formal power series
$\sum_{k=1}^\infty c_k\mathbb D_k\xi^{\otimes k}=\sum_{k=1}^\infty c_k\xi^k$
is quasi-holomorphic on a neighborhood of zero. By \eqref{e5w5qy3wdd}, this implies that the function $\rho(\xi)=\exp[C(\xi)]$ is holomorphic on a neighborhood of zero. 
\end{proof}

Let us now present some examples of lifted Sheffer sequences that are important for different applications of infinite dimensional analysis, and that satisfy the assumption of Proposition \ref{xresw5yq}. For more details and references, see \cite{FKLO}. 

\begin{example}[{\it Falling and rising factorials on $\Phi'$}] \label{es5w5w3}
The  falling factorials form the sequence of monic polynomials on $\mathbb C$ of binomial type that are explicitly given by
$$s_n(z)=(z)_n:=z(z-1)(z-2)\dotsm (z-n+1).$$ 
The generating function of the falling factorials is
$$G(z,u)=\exp[z\log(1+u)]=(1+u)^z.$$
The corresponding lifted sequence of polynomials  of binomial type on $\Phi'$ has the generating function
$$
\sum_{n=0}^\infty \frac1{n!}\,\langle S^{(n)}(\omega),\xi^{\otimes n}\rangle=\exp\big[\langle\omega,\log(1+\xi)\rangle\big].$$

Similarly,  the rising factorials on $\mathbb C$ form the sequence of monic polynomials on $\mathbb C$ of binomial type that are explicitly given by
$$s_n(z)=(z)^n:=(-1)^n(-z)_n=z(z+1)(z+2)\dotsm (z+n-1),$$ 
and their generating function is 
$$
\sum_{n=0}^\infty\frac{u^n}{n!}s_n(z)=\exp[-z\log(1-u)]=(1-u)^{-z}.$$
The corresponding rising factorials on $\Phi'$ are of binomial type and have the generating function
$$\sum_{n=0}^\infty \frac1{n!}\,\langle S^{(n)}(\omega),\xi^{\otimes n}\rangle=\exp\big[\langle\omega,-\log(1-\xi)\rangle\big].$$
\end{example}

\begin{remark}
In the theory of point processes, the so-called $K$-transform is utilized  to define the correlation measure of a point process, see e.g.\ \cite{KK}. This operator transforms functions defined on the space of finite configurations into functions defined on the space of infinite configurations in $\R^d$. In fact, the $K$-transform can also be understood as the umbral operator $\mathfrak U$ corresponding to the falling factorials on $\Phi'$. Under this interpretation, the statement that the umbral operator $\mathfrak U$ extends to a linear self-homeomorphism of  $\mathcal E^{1}_{\mathrm{min}}(\Phi')$ becomes a stronger result than \cite[Theorem 5.1]{KKO2}.
  
\end{remark}

\begin{example}[{\it Charlier polynomials on $\Phi'$}]
The classical Charlier polynomials on $\mathbb C$ (or rather $\R$) is a Sheffer sequence with the generating function
$$\sum_{n=0}^\infty\frac{u^n}{n!}s_n(z)=\exp [z\log(1+u)-u].$$
Its lifting is the sequence of Charlier polynomials on $\Phi'$ and it has the generating function
$$\sum_{n=0}^\infty \frac1{n!}\,\langle S^{(n)}(\omega),\xi^{\otimes n}\rangle=\exp\bigg[\langle\omega,\log(1+\xi)\rangle-\int_{\R^d}\xi(x)\,dx\bigg].$$
These polynomials play a fundamental role in Poisson analysis as they are orthogonal with respect to Poisson random measure on $\R^d$. 
Note that the falling factorials on $\Phi'$ is the basic sequence for the Charlier polynomials on $\Phi'$.
\end{example}

\begin{example}[{\it Laguerre polynomials on $\Phi'$}]\label{54wu3}
The (monic) Laguerre polynomials on $\mathbb C$ (or rather $\R$), corresponding to a parameter $k\ge-1$, have the generating function 
$$
\sum_{n=0}^\infty\frac{u^n}{n!}s_n(z)=\exp\bigg[\frac{zu}{1+u}\bigg](1+u)^{-(k+1)}=\exp\bigg[\frac{zu}{1+u}-(k+1)\log(1+u)\bigg] .$$ Its lifting gives a sequence of the Laguerre polynomials on $\Phi'$ which has the generating function
$$\sum_{n=0}^\infty \frac1{n!}\,\langle S^{(n)}(\omega),\xi^{\otimes n}\rangle=\exp\bigg[\bigg\langle\omega,\frac{\xi}{1+\xi}\bigg\rangle-(k+1)\int_{\R^d}\log(1+\xi(x))\,dx\bigg].$$
In particular, for $k=-1$, this  sequence is of binomial type. For $k=0$, these polynomials are orthogonal with respect to the gamma random measure on $\R^d$. 
\end{example}

\begin{remark} Note that, in view of Proposition \ref{6ew6u47eies}, for all the Sheffer polynomials on $\mathbb C$ appearing in Examples~\ref{es5w5w3}--\ref{54wu3}, the corresponding Sheffer  operator $\mathfrak S$ cannot be extended to $\mathfrak S\in\mathcal L(\mathcal E^{\alpha}_{\mathrm{min}}(\mathbb C))$ with $\alpha>1$.
\end{remark}

\subsection*{Acknowledgments}
The authors are grateful to the anonymous referee for valuable comments that improved the manuscript. 
MJO and LS were supported by the Portuguese 
national funds through FCT--Funda{\c c}\~ao para a Ci\^encia e a Tecnologia, 
within the project UID/MAT/04561/2019.

\appendix
\renewcommand{\thesection}{A}
\setcounter{equation}{0}
\setcounter{theorem}{0}

\section*{Appendix: Proofs of propositions in Section \ref{vrtew6u3}}

\begin{proof}[Proof of Proposition \ref{vytde6e6}]
 Let $\tau\in T$. Choose $\tau'\in T$ such that $\mathcal H_{\tau'}\subset\mathcal H_{\tau}$ and the embedding operator is of Hilbert--Schmidt class. 
Since the function $f\restriction_{\mathcal H_{-\tau'}}:\mathcal H_{-\tau'}\to\mathbb C$ is entire, it can be represented in the form
%\begin{equation}\label{trdsd}
$f(\omega)=f(0)+\sum_{n=1}^\infty \psi^{(n)}(\omega,\dots,\omega)$. 
Here, for each $n\in\mathbb N$, $\psi^{(n)}:\mathcal H_{-\tau'}^n\to\mathbb C$ is a symmetric bounded $n$-linear form, see \cite[Section~3.1]{Dineen2}. Obviously, we can identify $\psi^{(1)}$ with $\varphi^{(1)}\in\mathcal H_{\tau'}$ , so that $\psi^{(1)}(\omega)=\langle\omega,\varphi^{(1)}\rangle$ for all $\omega\in\mathcal H_{-\tau'}$. For each $n\ge 2$, by using the Kernel Theorem, see e.g.\ \cite[Chapter 1, Theorem 2.3]{BK}, we conclude that there exists a unique $\varphi^{(n)}\in\mathcal H_{\tau}^{\odot n}$ such that
$$\psi^{(n)}(\omega,\dots,\omega)=\langle \omega^{\otimes n},\varphi^{(n)}\rangle,\quad \omega\in\mathcal H_{-\tau}\subset\mathcal H_{-\tau'}.$$
Since $\tau\in T$ was arbitrary, we have that, for each $n\in\mathbb N$, $\varphi^{(n)}\in\bigcap_{\tau \in T}\mathcal H_\tau^{\odot n}=\Phi^{\odot n}$. From here the statement of the proposition follows. 
\end{proof}

\begin{proof}[Proof of Proposition \ref{6ei650}]
Without loss of generality, we may assume that $\mathcal G$ is the complexification of a real Hilbert space $\mathcal G_{\R}$, and for any $g_1,g_2\in\mathcal G$, we denote $\langle g_1,g_2\rangle:=(g_1,\overline{g_2})_{\mathcal G}$.

 Since the function $F$ is holomorphic, by \cite[Section~3.1]{Dineen2}, there exist  $U'$, an open neighborhood of zero in $\Phi$ that is a subset of $U$, and a symmetric continuous $n$-linear mapping $\Psi^{(n)}:\Phi^n\to \mathcal G$ ($n\in\mathbb N$)  
 such that, for all $\xi\in U'$,  
%\begin{equation}\label{cre6u4}
$F(\xi)=F(0)+\sum_{n=1}^\infty \Psi^{(n)}(\xi,\dots,\xi)$,
%\end{equation}
where the series converges uniformly on $U'$ in the $\mathcal G$ space. In particular, $\|F(\cdot)\|_{\mathcal G}$ is bounded on $U'$. 
Let $C_{18}:=\sup_{\xi\in U'}\|F(\xi)\|_{\mathcal G}$.

Choose $\tau'\in T$ and $r>0$ such that 
$\{\xi\in\Phi\mid \|\xi\|_{\tau'}< 2r\}\subset U'$.
Fix any
$\xi\in\Phi$ with  $\|\xi\|_{\tau'}=r$ and $g\in\mathcal G$. Consider the holomorphic function
$$\{z\in\mathbb C\mid |z|<2\}\ni z\mapsto  \langle F(z\xi),g\rangle=\langle F(0),g\rangle+\sum_{n=1}^\infty z^n \langle \Psi^{(n)}(\xi,\dots,\xi),g\rangle\in\mathbb C.$$
By applying Cauchy's Integral Formula on the disk $\{z\in\mathbb C: |z|\le 1\}$,  we get 
$$|\langle \Psi^{(n)}(\xi,\dots,\xi),g\rangle|\le C_{18}\|g\|_{\mathcal G},\quad n\in\mathbb N.$$
Therefore, for all $\xi\in\Phi$ and all $g\in\mathcal G$, 
\begin{equation}\label{ytwdfei62r}
|\langle \Psi^{(n)}(\xi,\dots,\xi),g\rangle|\le C_{18}\frac{\|\xi\|_{\tau'}^n}{r^n}\,\|g\|_{\mathcal G}.
\end{equation}
By using a consequence of the polarization formula on a Hilbert space (see  \cite[Proposition~1.44]{Dineen2}), we conclude from \eqref{ytwdfei62r} that, for any $\xi_1,\dots,\xi_n\in\Phi$ and $g\in\mathcal G$,
\begin{equation}
|\langle \Psi^{(n)}(\xi_1,\dots,\xi_n),g\rangle|\le C_{18}\, \frac{1}{r^n }\,\|\xi_1\|_{\tau'}\dotsm \|\xi_n\|_{\tau'}\|g\|_{\mathcal G}.\label{tuyqdwr5i}
\end{equation}
 In particular, formula \eqref{tuyqdwr5i} implies  that $\langle \Psi^{(n)}(\xi_1,\dots,\xi_n),g\rangle$ extends  to  a bounded $(n+1)$-linear form on 
 $\mathcal H_{\tau'}^n\times\mathcal G$ that is symmetric in the first $n$ variables.

Let $\tau \in T$ be such that $\mathcal H_{\tau}\subset\mathcal H_{\tau'}$ and the embedding operator is of Hilbert--Schmidt class. 
By using the Kernel Theorem,  we conclude that, for each $n\in\mathbb N$, there exists $\Theta^{(n)}\in\mathcal H_{-\tau}^{\odot n}\otimes\mathcal G$ such that, for all $\xi_1,\dots,\xi_n\in\Phi$ and $g\in \mathcal G$,
\begin{equation*}
 \langle \Psi^{(n)}(\xi_1,\dots,\xi_n),g\rangle=\langle\Theta^{(n)},\xi_1\odot\dots\odot\xi_n\otimes g\rangle .\end{equation*}
 Furthermore, it follows from \eqref{tuyqdwr5i} and the proof of the Kernel Theorem that, for all $n\in\mathbb N$,
\begin{equation}\label{tf7er58}
\|\Theta^{(n)}\|_{\mathcal H_{-\tau}^{\odot n}\otimes\mathcal G}\le 
C_{18} \frac{c_{\tau,\tau'}^n}{r^n },\end{equation}
where $c_{\tau,\tau'}$ is the Hilbert--Schmidt norm of the operator of embedding of $\mathcal H_\tau$ into $\mathcal H_{\tau'}$. 

For each $\varphi^{(n)}\in\mathcal H_\tau^{\odot n}$, we denote by $\langle \Theta^{(n)},\varphi^{(n)}\rangle$ the element of $\mathcal G$ that satisfies, for all $g\in\mathcal G$, 
\begin{equation}\label{utf7r87o}
\big\langle \langle \Theta^{(n)},\varphi^{(n)}\rangle,  g\big\rangle=\langle \Theta^{(n)},\varphi^{(n)}\otimes   g\rangle.
\end{equation}
Now, for each $n\in\mathbb N$, we define 
$A_n\in \mathcal L(\mathcal H_\tau^{\odot n},\mathcal G)$ by the formula 
$A_n\varphi^{(n)}:=\langle \Theta^{(n)},\varphi^{(n)}\rangle$ for 
$\varphi^{(n)}\in\mathcal H_{\tau}^{\odot n}$.
By \eqref{tf7er58},   and \eqref{utf7r87o},
\begin{equation*}%\label{r6we6x}
\|A_n\|_{\mathcal L(\mathcal H_\tau^{\odot n},\mathcal G)} \le C_{18} \frac{c_{\tau,\tau'}^n}{r^n }.\end{equation*}
From here the statement of the proposition follows.
\end{proof}

\begin{proof}[Proof of Theorem \ref{tctfdr6tqde6}] We start with 

\begin{lemma}\label{yqdi5}
Let $\alpha>0$, $\tau\in T$, and $l\in\mathbb N_0$.  Then there exists   $l'\in\mathbb N_0$ such that
$\mathbf E^\alpha_{l'}(\mathcal H_{-\tau})\subset \mathcal E_l^\alpha(\mathcal H_{-\tau})$, the embedding operator is continuous and  for all $f\in \mathbf E^\alpha_{l'}(\mathcal H_{-\tau})$
\begin{equation}\label{vy6e84}
\mathbf n_{\tau,l,\alpha}(f)\le  \|f\|_{\tau,l',\alpha}.
\end{equation}
\end{lemma}

\begin{proof} Let $f(\omega)$ be of the form \eqref{t7red6q5}.  Let $r\in\mathbb N$ and $p,q\in(1,\infty)$ with $\frac1p+\frac1q=1$. Furthermore, we assume $ q\ge\alpha$. 
By, using the H\"older inequality, we get 
\begin{align}
|f(\omega)|&\le\sum_{n=0}^\infty\|\omega\|_{-\tau}^n\|\varphi^{(n)}\|_\tau \frac{2^{rn}(n!)^{1/\alpha}}{2^{rn}(n!)^{1/\alpha}}\notag\\
&\le \left(\sum_{n=0}^\infty \|\varphi^{(n)}\|_\tau^p 2^{rnp}(n!)^{p/\alpha} \right)^{1/p} \bigg(\sum_{n=0}^\infty\left(
\|\omega\|_{-\tau}^{n\alpha}\frac1{2^{rn\alpha}}\frac1{n!}
\right)^{q/\alpha}\bigg)^{1/q}\notag\\
&\le \bigg(\sum_{n=0}^\infty \|\varphi^{(n)}\|_\tau 2^{rn}(n!)^{1/\alpha} \bigg) \bigg(\sum_{n=0}^\infty
\|\omega\|_{-\tau}^{n\alpha}\frac1{2^{rn\alpha}}\frac1{n!}
\bigg)^{1/\alpha}\notag\\
&= \|f\|_{\tau,r,\alpha}\exp\left[\frac1{\alpha 2^{r\alpha}}\,\|\omega\|_{-\tau}^\alpha\right]. \label{s5w7} 
\end{align}
For the given $\alpha>0$ and $l\in\mathbb N_0$, choose $r\in\mathbb N$ such that $2^l\le \alpha 2^{r\alpha}$. Let $l':=r$. Then, by \eqref{d6e7i4} and \eqref{s5w7}, formula \eqref{vy6e84} holds.\end{proof}

\begin{lemma}\label{rte684as}
Let $\alpha>0$, $\tau\in T$, and $l\in\mathbb N_0$. Let $\tau'\in T$ be such that the operator of embedding of $\mathcal H_{\tau'}$ into $\mathcal H_\tau$ is of Hilbert--Schmidt class. Then there exists $l'\in\mathbb N_0$ such that
$\mathcal E^\alpha_{l'}(\mathcal H_{-\tau'})\subset \mathbf E^\alpha_l(\mathcal H_{-\tau})$,
and the embedding operator is continuous, i.e., there exists $C_{19}>0$ such that, for all $f\in \mathcal E^\alpha_{l'}(\mathcal H_{-\tau'})$,    
\begin{equation}\label{yqefd6urd}
\|f\|_{\tau,l,\alpha}\le C_{19}\, \mathbf n_{\tau',l',\alpha}(f).\end{equation}
\end{lemma}

\begin{remark}
Since $\mathcal H_{\tau'}\subset \mathcal H_\tau$, we get $\mathcal H_{-\tau}\subset\mathcal H_{-\tau'}$. So Lemma~\ref{rte684as} states, in particular, that  each function from $\mathcal E^\alpha_{l'}(\mathcal H_{-\tau'})$ restricted to $\mathcal H_{-\tau}$ belongs to $\mathbf E^\alpha_l(\mathcal H_{-\tau})$. 
\end{remark}

\begin{proof}[Proof of Lemma \ref{rte684as}]  The proof is partially similar to the proof of Proposition~\ref{6ei650}.
Let $f:\mathcal H_{-\tau'}\to\mathbb C$ be an entire function and let  $f$ be of the form \eqref{t7red6q5}. For a fixed $\omega\in\mathcal H_{-\tau'}$, $\|\omega\|_{-\tau'}=1$, consider the entire function
$\mathbb C\ni z\mapsto f(z\omega)=\sum_{n=0}^\infty z^n \langle\omega^{\otimes n},\varphi^{(n)}\rangle$. By applying Cauchy's Integral Formula on the disk $\{z\in\mathbb C: |z|\le\rho_n\}$ with $\rho_n>0$,
 and using \eqref{d6e7i4}, we obtain 
$$ |\langle\omega^{\otimes n},\varphi^{(n)}\rangle|\le 
\frac{1}{\rho^n_n}\,\mathbf n_{\tau',l',\alpha}(f)\exp(2^{-l'}\rho_n^\alpha),\quad l'\in\mathbb N_0.$$
Hence, for all $\omega\in \mathcal H_{-\tau'}$,
\begin{equation*}
|\langle\omega^{\otimes n},\varphi^{(n)}\rangle|\le 
\frac{\|\omega\|^n_{-\tau'}}{\rho^n_n}\,\mathbf n_{\tau',l',\alpha}(f)\exp(2^{-l'}\rho_n^\alpha),\quad l'\in\mathbb N_0.
\end{equation*}
By using a consequence of the polarization formula on a Hilbert space and the Kernel Theorem, we then get
\begin{equation}\label{rts5esw}
\|\varphi^{(n)}\|_{\tau} \le \frac{c_{\tau',\tau}^n}{\rho^n_n}\,\mathbf n_{\tau',l',\alpha}(f)\exp(2^{-l'}\rho_n^\alpha),\quad l'\in\mathbb N_0.
\end{equation}
Choosing $\rho_n=\left(2^{l'}n/\alpha\right)^{1/\alpha}$ and using the estimate $n!\le n^n$, we get from \eqref{rts5esw}
\begin{align*}
\|\varphi^{(n)}\|_{\tau} &\le \frac1{(n!)^{1/\alpha}}\left(\frac{(\alpha e)^{1/\alpha}c_{\tau',\tau}}{2^{l'/\alpha}}\right)^n\mathbf n_{\tau',l',\alpha}(f).
\end{align*}
From here we easily conclude that inequality \eqref{yqefd6urd} holds for a sufficiently large $l'$.
\end{proof}

Finally, the statement of Theorem~\ref{tctfdr6tqde6} follows from Lemmas~\ref{yqdi5} and \ref{rte684as}.
\end{proof}

\end{document}